\documentclass[12pt]{amsart}
\usepackage{amssymb}
\usepackage{amsmath}
\usepackage{amsthm}
\usepackage{wrapfig, subfigure,graphicx}
\usepackage{amsmath,amsfonts,latexsym,graphicx,amssymb,url}
\usepackage{amsmath,txfonts,pifont,bbding,pxfonts,manfnt}
\usepackage{psfrag}
\usepackage{hyperref}
\usepackage[active]{srcltx}
\usepackage{pdfsync}
\usepackage[none]{hyphenat}
\newcommand{\dist}{\text{dist}} 
\newcommand{\diam}{\text{diam}}

\newcommand{\supp}{\text{{\rm supp }}}

\newcommand{\R}{{\mathbb R}} 

\def \p {\partial}

\def\p{\rho}

\theoremstyle{plain}
\newtheorem{theorem}{Theorem}[section]
\newtheorem{cor}[theorem]{Corollary}
\newtheorem{lem}[theorem]{Lemma}
\newtheorem{prop}[theorem]{Proposition}
\newtheorem{defi}[theorem]{Definition}
\newtheorem{remark}[theorem]{Remark}

\providecommand{\bysame}{\makebox[3em]{\hrulefill}\thinspace}

\begin{document}

\setcounter{equation}{0}










\title[The Reflector Problem and the inverse square law]
{The Reflector Problem and the inverse square law}
\author[C. E. Guti\'errez and A. Sabra]
{Cristian E. Guti\'errez\\
 and \\
Ahmad Sabra}
\begin{abstract}
We introduce a model to design reflectors that take into account the inverse square law for radiation.
We prove existence of solutions  both in the near and far field cases, when the input and output energies are prescribed.
\end{abstract}
\thanks{\today.}
\thanks{Authors partially supported by NSF grant DMS--1201401.
We thank the Karlsruhe Institute of Technology, Germany, for the support during March and April 2013 where this paper was completed. It is also a pleasure to thank Professors Christian Koos and Wolfgang Reichel for their kind hospitality.}
\address{Department of Mathematics\\Temple University\\Philadelphia, PA 19122}
\email{gutierre@temple.edu, ahmad.sabra@temple.edu}

\maketitle


\setcounter{equation}{0}
\section{Introduction}

It is known that the intensity of radiation is inversely proportional to the square of the distance from the source.
In particular, at large distances from the source, the radiation intensity is distributed over larger surfaces and therefore
the intensity per unit area decreases as the distance from the surface to the source increases.

In this paper, we solve a problem in radiometry involving the inverse square law, see e.g., \cite{book:born-wolf}[Section 4.8.1, formula (10)] and \cite[Chapter 4]{ross:introtoradiometryandphotometry}.
We begin explaining the concepts needed to pose the problem.
Suppose that $f(x)$ denotes the radiant intensity in the direction $x\in S^{2}$, measured in Watts per steradian; in other words, we are assuming the radiation emanates from the origin in a non isotropic fashion.
Here $S^2$ denotes the spherical surface of radius one centered at the origin in $\R^3$.
If
we have a piece of surface $\Omega\subseteq S^{2}$, then the total amount of energy received at $\Omega$, or in other words the radiant flux through $\Omega$, is given by the surface integral
\[
\Phi=\int_\Omega f(x)\,dx\footnote{The units for this quantity are Watts because the units for $\Omega\subseteq S^2$ are considered non dimensional units, i.e., $\Omega$ is measured in steradians.}.
\]
The {\it irradiance} $E$ is the amount of energy or radiant flux incident on a surface per unit area; it is measured in $W/m^2$.
 We notice that in this case radiation flows perpendicularly through $\Omega$.
If the piece of surface $\Omega\subseteq S^2$ is dilated by $r$ with respect to the origin, then the new surface $r\Omega$ is contained in the sphere of radius
$r$ centered at the origin. If there is no loss, then the radiant flux on $\Omega$ and the
radiant flux on $r\Omega$ must be the same. Therefore, if $r\Omega$ has a larger area than $\Omega$, i.e., $r>1$, the average amount of energy per unit area at $r\Omega$ must be smaller than the average amount of energy per unit area at $\Omega$.
Likewise if $r<1$, then the energy per unit area in $r\Omega$ gets larger.
Let $\Phi_1,\Phi_r$ be the radiant fluxes at $\Omega$ and $r\Omega$, respectively.
As we said, by energy conservation $\Phi_1=\Phi_r$. If $|\Omega|$ denotes the surface area of $\Omega$, we have that the irradiance in $\Omega$ is $E_1=\dfrac{\Phi_1}{|\Omega |}$,
and the irradiance at $r\Omega$ is $E_r= \dfrac{\Phi_r}{|r\Omega |}$. Therefore
\begin{equation}\label{eq:inversesquarelaw}
E_r=\dfrac{\Phi_1}{|r\Omega|}=\dfrac{\Phi_1}{r^2|\Omega|}=\dfrac{E_1}{r^2}
=\dfrac{1}{|\Omega|}\int_\Omega \dfrac{f(x)}{r^2}\,dx.
\end{equation}
Suppose we break the domain $\Omega=\cup_{j=1}^N\Omega_j$, with $\Omega_j$ disjoint, and such that
$f$ is approximately equal to a constant $f_j$ on $\Omega_j$. Then the radiant flux $\Phi$ received at $\Omega$ satisfies $\Phi=\sum_{j=1}^N \Phi_j$ where $\Phi_j$ is the radiant flux at $\Omega_j$.
We have $\Phi_j=(\Phi_j)_r$ and the irradiance $E_j$ in $\Omega_j$ is $\dfrac{\Phi_j}{|\Omega_j |}$, and the irradiance $E_j^r$ in $r\Omega_j$ is $\dfrac{(\Phi_j)_r}{|r\Omega_j |}$. So, as before,
\[
E_j^r=\dfrac{E_j}{r^2}=\dfrac{1}{|\Omega_j|}\int_{\Omega_j} \dfrac{f(x)}{r^2}\,dx\approx \dfrac{f_j}{r^2}=\dfrac{f(x_j)}{r^2};\quad x_j\in \Omega_j.
\]
We then  have that
\[
\int_\Omega f(x)\,\dfrac{1}{r^2}\,dx
=
\sum_{j=1}^{N}\int_{\Omega_j} f(x)\,\dfrac{1}{r^2}\,dx
\approx
\sum_{j=1}^{N} f(x_j)\,\dfrac{1}{r^2}\,|\Omega_j|
\approx
\sum_{j=1}^{N} E_j^r\,|\Omega_j|
\approx
\int_\Omega E^r(x)\,dx,
\]
where $E^r(x)$ represents the irradiance over an infinitesimal patch over $r\Omega$ around the direction $x$.


Now, instead of considering a patch of surface on $r\Omega$, we look at a sufficiently smooth surface given parametrically by $\rho(x)x$ for $x$ on a piece of $S^2$.
We fix a unit direction $y$, and take an infinitesimally small neighborhood of $y$ in $S^2$, say $B_\varepsilon(y)$.
We want to calculate approximately the area of the piece of surface described by $A(y)=\{\rho(x)x:x\in B_\varepsilon(y)\}$,
and the irradiance over $A(y)$.
The tangent plane $T$ at $\rho(y)y$ has normal unit $\nu(y)$. The region on the tangent plane for $x\in B_\varepsilon(y)$
is denoted by $T(y)=\{z\in T:z/|z|\in B_\varepsilon(y)\}$. Let us calculate the area of $T(y)$.
Notice $A(y)$ and $T(y)$ have approximately the same area.
Consider the plane $T'$ perpendicular to the direction $y$ passing through $\rho(y)y$,
and denote by $T'(y)$ the region on that plane for $x\in B_\varepsilon(y)$, i.e.,
$T'(y)=\{z\in T':z/|z|\in B_\varepsilon(y)\}$.
We have that
\[
|T(y)|\approx \dfrac{|T'(y)|}{ y\cdot \nu(y)}.
\]
Therefore
\[
|A(y)|\approx \dfrac{|T'(y)|}{ y\cdot \nu(y)}.
\]
On the other hand, by similarity
\[
|T'(y)|\approx \rho(y)^2 |B_\varepsilon(y)|,
\]
and we then obtain that
\[
|A(y)|\approx \dfrac{\rho(y)^2 |B_\varepsilon(y)|}{ y\cdot \nu(y)}.
\]
The irradiance on the infinitesimal area $A(y)$ then equals
\[
\dfrac{\text{radiant flux over $A(y)$}}{|A(y)|}\approx
\dfrac{\int_{B_\varepsilon(y)} f(x)\,dx}{\dfrac{\rho(y)^2 |B_\varepsilon(y)|}{ y\cdot \nu(y)}}
\approx
\dfrac{f(y)\, y\cdot \nu(y)}{\rho(y)^2}.
\]
%
Therefore, if we have a piece of surface $\sigma$ in $\R^3$ parameterized by a function $\rho(x)x$ for $x \in \Omega$ with $\Omega\subseteq S^2$, and we have radiant intensity $f(x)$ for each $x\in \Omega$,
then based on the above considerations we introduce the following quantity, measuring
the amount of irradiance at the patch of surface $\sigma(\Omega)$, by the surface integral
\begin{equation}\label{enu}
\int_\Omega f(x)\, \dfrac{x\cdot \nu_\rho(x)}{\rho(x)^2}\,dx,
\end{equation}
where $\nu_\rho(x)$ is the outer unit normal to the surface $\sigma$ at the point $\rho(x)x$.
Notice that if $\sigma$ is far from the source, then
$\rho$ is very large and therefore the irradiance on $\sigma$ is very small. Formula \eqref{enu} generalizes
the inverse square law in radiometry from \cite{book:born-wolf}[Section 4.8.1, formula (10)].

The problem we propose and solve in this paper is the following.
Suppose $f$ is a positive function given in $\Omega\subseteq S^2$
and $\eta$ is a Radon measure on a bounded set $D$ contained on a surface in $R^3$,
and $\dist(0,D)>0$.
Suppose radiation emanates from the origin with radiant intensity $f(x)$ in each direction $x\in \Omega$.
We want to find a reflector surface $\sigma$ parameterized by $\rho(x)x$ for $x\in \Omega$ such that the radiation is reflected off by $\sigma$ into the set $D$ and such that the irradiance received on each patch of $D$ has at least measure  $\eta$. In other words, we propose to find $\sigma$ such that
\begin{equation} \label{eq:pbequ}
\int_{\tau_\sigma (E)}f(x)\, \dfrac{x\cdot \nu(x)}{\rho(x)^2}\,dx\geq \eta(E),
\end{equation}
for each $E\subseteq D$, where the set $\tau_\sigma(E)$ is the collection of directions $x\in \Omega$ that $\sigma$ reflects off in $E$.
We ask the reflector to cover all the target $D$, that is, $\tau_\sigma(D)=\Omega$. In particular, from \eqref{eq:pbequ} we need to have
\begin{equation} \label{eq:admissiblereflector}
\int_{\Omega}f(x)\, \dfrac{x\cdot \nu(x)}{\rho(x)^2}\,dx\geq \eta(D),
\end{equation}
and we say in this case that the reflector $\sigma$ is admissible.
Since $f$ and $\eta$ are given but we do not know the reflector $\sigma$, we do not know a priori if \eqref{eq:admissiblereflector} holds. However, assuming that
the input and output energies satisfy
\[
\int_{\Omega}f(x)\,dx\geq \dfrac{1}{C}\eta(D),
\]
where $C$ is an appropriate constant depending only on the distance between the farthest point on the target and the source, and
from how close to the source we want to place the reflector,
we will prove that there exists a reflector $\sigma$ satisfying \eqref{eq:pbequ}; see condition \eqref{probcond} and Theorems \ref{discretethmgeq} and
\ref{thm:solutionforgeneralmeasure}.
In particular, we will see that if the target $D$ has a point very far away from the source, then the constant $1/C$ will be very large and therefore, for a given $\eta$ we will need more energy $f$ at the outset to prove the existence of a reflector satisfying
\eqref{eq:pbequ}.
We will also see that, in general, for each fixed point $P_0$ in the support of the measure $\eta$,
we construct a reflector that overshoots energy only at $P_0$, that is, for each set $E\subseteq \Omega$ such that $P_0\notin E$ we have equality in \eqref{eq:pbequ}; see Theorems
\ref{discretethmgeq}[parts (2) and (3)] and Theorem \ref{thm:solutionforgeneralmeasure}.
In Subsections \ref{sub:discreteovershooting} and \ref{sub:overshooting}, we show that it is possible to construct a reflector that minimizes the overshooting at $P_0$, that is unique in the discrete case, see Theorem \ref{uniqueness}.

To solve our problem, we introduce the notion of reflector and reflector measure with supporting ellipsoids of revolution, and show that \eqref{eq:pbequ} makes sense in terms of Lebesgue measurability,
Section \ref{sec:reflectorsandreflectormeasures}. With this definition, reflectors are concave functions and therefore differentiable a.e., so the normal $\nu(x)$ exists a.e..
To obtain the $\sigma$-additivity of the reflector measure given in Proposition \ref{mumeasure}, we need to assume that the target $D$ is contained in a hyperplane or $D$ is countable. This is needed in the proof of Lemma \ref{lem:aleksandrovforreflectors} and Remark \ref{rmk:countabletarget}, a result that might fail otherwise, see Remark \ref{rmk:counterexampletoaleksandrov}.

With this definition of the reflector, the reflected rays might cross the reflector to reach the target, in other words, the reflector might obstruct the target in certain directions.
This physical issue can be avoided by assuming that the
supporting ellipsoids used in the definition of reflector have the target contained in their interiors.
Another kind of physical obstruction might happen when the target obstructs the incident rays in their way to the reflector.
All of these are discussed and illustrated in Subsection \ref{subset:physicalvisibilityissues}.

In Section \ref{sec:ellipsoids}, we show properties of ellipsoids of revolution that are needed in the proofs of our results.
In Section \ref{sec:Soldis}, the problem is first solved in the discrete case, that is, when the measure $\eta$ is concentrated on a finite target $D$. The solution for a general measure $\eta$
is given in Section \ref{sec:solutionforgeneralmeasure}.
Section \ref{sec:NearFieldEquation} contains the pde for the problem which is a Monge-Amp\`ere type equation.

In Section \ref{sec:Farfield}, we consider and solve the problem in the far field case, that is,
when $D$ is replaced by a set of directions in $\Omega^* \subseteq S^2$.

We finish this introduction mentioning results in the literature
that are relevant for this work and put our results in perspective.
The reflector problem in the far field case has been considered by L. Caffarelli, V. Oliker and X-J. Wang, see \cite{Caffarelli-Oliker:weakantenna}, \cite{Wang:antenna}, and \cite{Wang:antennamasstransport}. The near field case is in \cite{kochengin-oliker:nearfieldreflector} where the notion of reflector defined with supporting ellipsoids is introduced. In all these papers it is assumed that $\int_\Omega f(x)\,dx=\eta(\text{target})$, and the model does not take into account the inverse square law.
For the far field refraction, models taking into account the loss of energy due to internal reflection are considered in \cite{gutierrez-mawi:refractorwithlossofenergy}.

We believe this paper is the first contribution to the problem of constructing a reflector that takes into account how far it is from the source.

\setcounter{equation}{0}
\section{Ellipsoids}\label{sec:ellipsoids}

Let $\mathcal L$ be a plane in $\R^3$ and $P\in \R^3$ be a point not in $\mathcal L$.
An ellipsoid of revolution having a focus at $P$ is the locus of the points $X$ such that $|X-P|=\varepsilon \,\dist(X,\mathcal L)$, where $0<\varepsilon<1$
is a constant. $\mathcal L$ is called the directrix plane.
Notice that this determines the other focus which might not be the origin. Let $f:=\dist(P,\mathcal L)$ be the {\it focal parameter} of the ellipsoid.
If $O$ is the origin in $\R^3$, $P$ is an arbitrary point, and $c>|O-P|:=OP$, then
the ellipsoid of revolution $E(P)$ with foci $O$ and $P$, is given by
$\{X:|X|+|X-P|=c\}$.
The polar equation of such ellipsoid is
\begin{equation}\label{def:ofpolarradius}
\rho(x)=\dfrac{\dfrac{c^2-OP^2}{2c}}{1-\dfrac{OP}{c}\, x\cdot m},
\end{equation}
where $m=\dfrac{\overrightarrow{OP}}{OP}$, $x\in S^2$.
The {\it eccentricity} of $E(P)$ is by definition $\varepsilon=\dfrac{OP}{c}$, and we have $0<\varepsilon<1$.
We show that $\mathcal L$ is the plane perpendicular to the line with direction $m$ and having distance to the origin
 $c\dfrac{1+\varepsilon^2}{2\,\varepsilon}$. Indeed, let $M\in E(P)$ be the point where the distance between $\mathcal L$ and $E(P)$ is attained. We have $PM=\varepsilon \, \dist(M,\mathcal L)$, and $OP+2\,PM=c$, so $PM=\dfrac{c-OP}{2}=\dfrac{c}{2}(1-\varepsilon)$.
Then
 \begin{align*}
\dist(O,\mathcal L)&=OM+\dist(M,\mathcal L)=OP+PM+\dist(M,\mathcal L)\\
&= OP + \left(1+\dfrac{1}{\varepsilon}\right)PM=
c\,\varepsilon+\left(1+\dfrac{1}{\varepsilon}\right) \dfrac{c}{2}(1-\varepsilon)= c\,\dfrac{1+\varepsilon^2}{2\,\varepsilon},
\end{align*}
as desired.
We then have the following formula for the focal parameter:
$$
f=\dist(P,\mathcal L)= \dist(O,\mathcal L)-OP= \dfrac{c(1-\varepsilon^2)}{2\,\varepsilon}.
$$
We also have that
\begin{equation}\label{eq:formulaforthefocalparameter}
d:=\dfrac{c^2-OP^2}{2c}=f\,\varepsilon,
\end{equation}
where $d$ represents the radius of the circle obtained by intersecting the ellipsoid with a plane through $O$ perpendicular to $\overrightarrow{OP}$.

From the definition of $d$ we get $ 2cd=c^2- OP^2$, and dividing both sides of this equation by $c^2$ yields
$\varepsilon^2+\dfrac{2d}{OP}\varepsilon -1 =0$.
Solving for $\varepsilon$ we obtain:
\begin{equation} \label{epsequ}
\varepsilon=\sqrt{1+\dfrac{d^2}{OP^2}}-\dfrac{d}{OP}.
\end{equation}
Therefore, we shall denote by $E_d(P)$ the ellipsoid of foci $O$ and $P$ having polar equation
\begin{equation} \label{ellipse}
\rho_d (x)=\dfrac{d}{1-\varepsilon\, x \cdot m}.
\end{equation}

%
%
%
%
%
The ellipsoid is very eccentric if $\varepsilon\approx 1$, this means that $d$ is very close to zero and the ellipsoid degenerates to a segment. 
We will work with ellipsoids that are not too eccentric, that is, with $d/OP\geq \delta$, with $\delta>0$ arbitrarily fixed.
In other words, the eccentricity will arbitrary but controlled in advance.
However, if we take into account physical issues presented in Subsection \ref{subset:physicalvisibilityissues}, then we need to choose $\delta$ sufficiently large, see \eqref{eq:assumptionondeltatoavoidobstruction}.

We have the following simple proposition that will be used frequently in the paper.
\begin{prop}\label{propineq}
Let $O$ be the origin in $\R^3$ and $P\neq O$.
Fix $\delta>0$ and consider an ellipsoid $E_d(P)$ with $d \geq \delta\,OP$.
Then there exists a constant $0<c_\delta<1$, independent of $P$, such that $E_d(P)$ has eccentricity $\varepsilon\leq c_\delta$ and we have
\begin{equation}\label{impineq}
\dfrac{d}{1+c_{\delta}} \leq \min_{x \in S^2} \rho_d(x) \leq \max_{x \in S^2} \rho_d(x)\ \leq \dfrac{d}{1-c_\delta}.
\end{equation}
\end{prop}

\begin{proof}
From the definitions of $d$ and $\varepsilon$ we have
\begin{equation} \label{eqop}
OP=\dfrac{2\varepsilon d}{1-\varepsilon^2}.
\end{equation}
Since $d\geq \delta \,OP=\dfrac{2\,\delta\,\varepsilon \,d}{1-\varepsilon^2}$,
we have $\varepsilon^2+2\delta\varepsilon-1 \leq 0$.
Solving this quadratic inequality and keeping in mind that $\varepsilon > 0$ yields
\begin{equation}\label{epsbdd}
\varepsilon \leq -\delta+\sqrt{1+\delta^2}:=c_\delta<1.
\end{equation}
Then using equation (\ref{ellipse}) and noting that $1-\varepsilon \leq 1- \varepsilon\,  x\cdot m \leq 1+\varepsilon$,
we obtain
$$\dfrac{d}{1+c_{\delta}} \leq \dfrac{d}{1+\varepsilon} \leq \rho_d(x)\leq \dfrac{d}{1-\varepsilon}\leq \dfrac{d}{1-c_\delta}$$ for all $x \in S^2$.
\end{proof}

We recall the following proposition borrowed from \cite[Lemma 6]{kochengin-oliker:nearfieldreflector}.
\begin{prop}\label{propmonot}
For ellipsoids of fixed foci $O$ and $P$, the eccentricity $\varepsilon$ is a strictly decreasing function of $d$, and for each fixed $ x$, $\rho_d(x)$ is strictly increasing function of $d$.
\end{prop}



\begin{prop}\label{propnorm}
The outer unit normal $\nu_d$ to the ellipsoid $E_d(P)$ at the point $\rho_d(x)x$ is given by
\begin{equation}\label{nueq}
\nu_d(x)=\dfrac{x-\varepsilon\, m}{|x-\varepsilon\, m|}.
\end{equation}

\end{prop}

\begin{proof}
Let $y=\rho_d(x)\,x$, then $x=y/|y|$ and from equation (\ref{ellipse})
$|y|-\varepsilon\, m\cdot y =d$.
So the ellipsoid is a level set of the function $h(y)$=$|y| - \varepsilon\, m\cdot y$, and so by computing the gradient of $h$ the formula immediately follows.
\end{proof}

From the law of reflection we obtain the following well known fact.
\begin{prop}\label{snell}
Given an ellipsoid $E_d(P)$ with foci $O$, $P$,
each ray emitted from $O$ is reflected off to $P$.
\end{prop}

\setcounter{equation}{0}
\section{Reflectors and reflector measures}\label{sec:reflectorsandreflectormeasures}

In this section we introduce the definition of reflector and reflector measure, and prove some properties that will be used later in the paper.

\begin{defi}\label{def:definitionofreflector}
Let $\Omega\subseteq S^{2}$ such that $|\partial \Omega|=0$. The surface $\sigma=\{\rho(x)x\}_{x\in \bar{\Omega}}$ is a reflector from $\bar{\Omega}$ to $D$ if for each $x_0\in \bar{\Omega}$ there exists an ellipsoid $E_d(P)$ with $P\in D$ that supports $\sigma$ at $\rho(x_0)\,x_0$. That is, $E_d(P)$ is given by $\rho_d(x)=\dfrac{d}{1-\varepsilon \,x\cdot m}$ with $m=\dfrac{\overrightarrow{OP}}{OP}$ and satisfies $\rho(x)\leq \rho_d(x)$ for all $x\in \bar{\Omega}$ with equality at $x=x_0$.
\end{defi}
 Notice that each reflector is concave and therefore continuous.

 The reflector mapping associated with a reflector $\sigma$ is given by
 \[
 \mathcal N_\sigma(x_0)=\{P\in D: \text{there exists $E_d(P)$ supporting $\sigma$ at $\rho(x_0)x_0$}\};
 \]
 and the tracing mapping is
\[
\tau_\sigma(P)=\{x\in \bar{\Omega}: P\in \mathcal N_\sigma(x) \}.
\]

We prove the following lemma of crucial importance for the definition of reflector measure.
For its proof,  we use a real analysis result proved afterwards in Lemma \ref{lm:densitypointslebesgue}.

\begin{lem}\label{lem:aleksandrovforreflectors}
Suppose $D$ is contained in a plane $\Pi$ in $\R^3$, and let $\sigma$ be a reflector from $\bar{\Omega}$ to $D$.
Then the set
\[
S=\{x\in \bar{\Omega}:\text{there exist $P_1\neq P_2$ in $D$ such that $x\in \tau_\sigma(P_1)\cap \tau_{\sigma}(P_2)$}\}
\]
has measure zero in $S^{2}$.
\end{lem}

\begin{proof}
Let $N$ be the set of points where $\sigma$ is not differentiable. Since $\rho$ is concave, it is  locally Lipschitz and so the measure of $N$ in $S^{2}$ is zero. Let us write $S=(S\cap N)\cup (S\cap N^c)$. We shall prove that the measure, in $S^{2}$, of $S\cap N^c$ is zero.

Let $x_0\in S\cap N^c$, then there exist $E_{d_1}(P_1)$ and $E_{d_2}(P_2)$ supporting ellipsoids to $\sigma$ at $x_0$
with $P_1\neq P_2$,
and $x_0$ is not a singular point of $\sigma$. Then there is a unique normal $\nu_0$ to $\sigma$ at $x_0$ and
$\nu_0$ is also normal to both ellipsoids $E_{d_1}(P_1)$ and $E_{d_2}(P_2)$ at $x_0$.
From the Snell law applied to each ellipsoid we then get that
\[
r_i=x_0-2\,(x_0\cdot \nu_0)\,\nu_0,
\]
where $r_i=\dfrac{P_i-X_0}{|P_i-X_0|}$ is the reflected unit direction by the ellipsoid $E_{d_i}(P_i)$, with $X_0=\rho(x_0)x_0$,
$i=1,2$.
Therefore
\[
\dfrac{P_1-X_0}{|P_1-X_0|}=\dfrac{P_2-X_0}{|P_2-X_0|}.
\]
This implies that $X_0$ is on the line joining $P_1$ and $P_2$, and since $D\subseteq \Pi$ we get that
$E:=\{\rho(x_0)\,x_0:x_0\in S\cap N^c\}\subseteq \Pi$. That is, the graph of $\sigma$, for $x\in S\cap N^c$, is contained in the plane $\Pi$.

We will prove that the set $S\cap N^c$ has measure zero in $S^2$.

{\bf Case 1:} $O\in \Pi$.
We have $\rho(z)\,z\in \Pi$, for each $z\in S\cap N^c$, then incident ray with direction $z$ is contained in $\Pi$. This means that $S\cap N^c\subseteq \Pi$, and therefore
$S\cap N^c$ is contained in a great circle of $S^2$ and therefore has surface measure zero.

{\bf Case 2:} $O\notin \Pi$. In this case, $\Pi$ cannot be a supporting plane to $\sigma$ at any $x\in F:=S\cap N^c$. 
Otherwise, if $\Pi$ supports $\sigma$ at some $z\in F$, the ray with direction $z$ 
is reflected off at the point $\rho(z)z$ into a ray with
unit direction $r_z$ on the plane $\Pi$. If $\nu_0$ is the unit normal to $\Pi$, then $\nu_0\perp r_z$.
By the Snell law, $r_z=z-2\,(z\cdot \nu_0)\,\nu_0$, and so $\rho(z)z\cdot \nu_0=0$, and so $O\in \Pi$, contradiction.

To show that $F$ has measure zero we will use the notion of density point to a set.
%
%
Let $w(x)$ be the polar equation of the hyperplane $\Pi$.
We have that $\rho(x)=w(x)$ for all $x\in F$.
Let $z_0\in F$. Since $\rho$ is concave, there is a plane $H_{z_0}$, with polar radius $h_{z_0}(x)$, that supports $\sigma$ at $\rho(z_0)\,z_0$.
That is,
$h_{z_0}(x)\geq \rho(x)$ for all $x\in \bar{\Omega}$ and $h_{z_0}(z_0)=\rho(z_0)$. This plane is unique because $z_0$ is not a singular point of $\sigma$, and therefore $0\notin H_{z_0}$.
Since $\Pi$ is not a supporting plane, we have that $H_{z_0}\neq \Pi$, and then the sets
$A^+=\{x\in S^2:h_{z_0}(x)\geq w(x)\}$ and
$A_-=\{x\in S^2:h_{z_0}(x)<w(x)\}$ have both non empty interiors.

Consider the line $\Pi\cap H_{z_0}:=\ell$, which passes through $\rho(z_0)\,z_0$,
and the plane $J$ containing the line $\ell$ and the origin. The point $\rho(z_0)\,z_0\in J$.
For each $r>0$, the plane $J$ divides each spherical ball $B_r(z_0)$ into two disjoint pieces $B_r(z_0)^+$ and
$B_r(z_0)^-$ having the same spherical measure, 
$B_r(z_0)=B_r(z_0)^+\cup B_r(z_0)^-$,
and with $B_r(z_0)^+=B_r(z_0)\cap A^+$ and $B_r(z_0)^-
=B_r(z_0)\cap A^-$.
We have for $x\in F$, that $h_{z_0}(x)\geq \rho(x)=w(x)$, so $F\subseteq A^+$, and therefore $F\cap A^-=\emptyset$.
Then
\[
B_r(z_0)\cap F=(B_r(z_0)^+\cap F)\cup (B_r(z_0)^-\cap F)\subseteq B_r(z_0)^+
\]
and consequently we get
\[
|B_r(z_0)\cap F|_*\leq |B_r(z_0)^+|=\dfrac12 \,|B_r(z_0)|;
\]
$|\cdot |_*$ denotes the spherical Lebesgue outer measure in $S^2$ (notice that we do not know a priori if the set $F$ is measurable in the sphere).
That is,
\[
f(z_0):=\limsup_{r\to 0}\dfrac{|B_r(z_0)\cap F|_*}{|B_r(z_0)|}\leq 1/2.
\]
So
\[
z_0\in \{z\in F:f(z)<1\}:=M;
\]
but the set $M$ has measure zero by Lemma \ref{lm:densitypointslebesgue}. 
Therefore the measure of $F$ is zero in $S^2$ and the proof of the lemma is complete.
%


\end{proof}

\begin{remark}\label{rmk:countabletarget}\rm
If $D$ is a finite or countable set, then the conclusion of Lemma \ref{lem:aleksandrovforreflectors} holds
regardless if $D$ is on a plane. In fact, let $D=\{P_i\}_{i=1}^\infty$, with $O\notin D$.
Let $\Pi_{ij}$ be the plane (or line) generated by the points $O,P_i,P_j$.
Following the proof of Lemma \ref{lem:aleksandrovforreflectors} we have that 
\[
S\cap N^c\subseteq \cup_{i\neq j}\Pi_{ij},
\]
and since the surface measure of $S\cap N^c\cap \Pi_{ij}$ is zero we are done.
\end{remark}

\begin{remark}\label{rmk:counterexampletoaleksandrov}\rm
We present an example of a target $D$ that is not contained in a plane, a set $\Omega\subseteq S^2$, and a reflector $\sigma$ from $\bar \Omega$ to $D$ such that the set $S$ in Lemma \ref{lem:aleksandrovforreflectors} has positive measure.
Consider the origin $O$, the point $P_0=(0,2,0)$, and the half sphere $S_-=\{X=(x_1,x_2,x_3):|X-P_0|=1,x_3\leq 0\}$.
Let $D= S_- \cup \{P_0\}$, and consider the ellipsoid $E_d(P_0)=\{\rho_d(x)x\}_{x\in S^2}$ with $d$ large enough such that it contains $D$. Let $\Omega=\{x=(a,b,c)\in S^2:c>0\}$ and the reflector $\sigma=\{\rho_d(x)x\}_{x \in \bar{\Omega}}$.

Each point $P\in D$ is reached by reflection, because the ray from $P_0$ passing through $P$ intersects $\sigma$ at some point $P'$ and since $\sigma$ is an
ellipsoid the ray emanating from $O$ with direction $P'/|P'|$ is reflected off to $P$.
We can see in this case that the set $S$ in Lemma \ref{lem:aleksandrovforreflectors} has positive measure. Indeed, let $P\in D$, $P\neq P_0$. By continuity there is an ellipsoid of revolution $E_{d'}(P)$ with foci $O$ and $P$ that supports $\sigma$ 
at some unit direction $x_P$. That is, $x_P\in \tau_\sigma(P)$. Since $\sigma$ is an ellipsoid we also have $x_P\in \tau_\sigma(P_0)$. Since there is a one to one correspondence between the points $P$ and $x_P$, the measure of $S$ is positive.

Similarly, Lemma \ref{lem:aleksandrovforreflectors} does not hold when $D$ is contained in a union of planes. For example, let $P=(0,2,0)$, $\mathcal C$ be the closed disk centered at $(0,2,-1)$ and radius 1 contained on the plane $z=-1$, and let the target be $D=P\cup \mathcal C$. If $\sigma$ is a sufficiently large ellipsoid with foci $O$, $P$, and containing $D$, then the set $\tau_\sigma(P)\cap \tau_\sigma(\mathcal C)$ has positive measure.
\end{remark}

\begin{lem}\label{lm:densitypointslebesgue}
Let $S\subseteq \R^n$ be a set not necessarily Lebesgue measurable and consider
\[
f(x):=\limsup_{r\to 0} \dfrac{|S\cap B_r(x)|_*}{|B_r(x)|},
\]
where $|\cdot |_*$ and $|\cdot |$ denote the Lebesgue outer measure and Lebesgue measure respectively.
If
\[
M=\{x\in S: f(x)<1\},
\]
then $|M|=0$.
Here $B_r(x)$ is the Euclidean ball centered at $x$ with radius $r$.

Moreover, if $B_r(x)$ is a ball in a metric space $X$ and
$\mu^*$ is a Carath\'eodory outer measure on $X$\footnote{From \cite[Theorem (11.5)]{wheeden-zygmund:book} every Borel subset of $X$ is Carath\'eodory measurable.}, then a similar result holds true for all $S\subseteq X$.
\end{lem}
\begin{proof}
We first assume $|S|_*<\infty$.
Fix $x\in M$. There exists a positive integer $m$ such that $f(x)<1-\dfrac{1}{m}<1$, and let $m_x$ be the smallest integer with this property.
So for each $\eta>0$ sufficiently small we have
\begin{equation}\label{densityinequality}
\sup_{0<r\leq \delta}\dfrac{|S\cap B_r(x)|_*}{|B_r(x)|}<1-\dfrac{1}{m_x},\qquad \text{for all $0<\delta \leq \eta$}.
\end{equation}
Given a positive integer $k$, let $M_k=\{x\in M: m_x=k\}$. We have $M=\cup_{k=1}^\infty M_k$.
We shall prove that $|M_k|=0$ for all $k$.
Suppose by contradiction that $|M_k|_*>0$ for some $k$, we also have that $|M_k|_*<\infty$.
Let us consider the family of balls $\mathcal F=\{B_r(x)\}_{x\in M_k}$ with $B_r(x)$ satisfying \eqref{densityinequality}. Then we have that the family $\mathcal F$ covers $M_k$ in the Vitali sense, i.e., for every $x\in M_k$ and for every $\eta>0$ there is ball in $\mathcal F$ containing $x$ whose diameter is less than $\eta$. Therefore from \cite[Corollary (7.18) and (7.19)]{wheeden-zygmund:book} we have that given $\varepsilon>0$ there exists a family of disjoint balls $B_1,\cdots ,B_N$ in $\mathcal F$ such that
\begin{align*}
|M_k|_* -\varepsilon &< \left| M_k\cap \cup_{i=1}^N B_i \right|,\qquad \text{and}\\
\sum_{i=1}^N |B_i| &< (1+\varepsilon)|M_k|_*.
\end{align*}
Since $M_k\subseteq S$, then from \eqref{densityinequality} we get
\[
|M_k|_* -\varepsilon< \left| S\cap \cup_{i=1}^N B_i \right|\leq \sum_{i=1}^N |S\cap B_i|\leq \left(1-\dfrac{1}{k}\right)\sum_{i=1}^N |B_i|
<(1+\varepsilon)\left(1-\dfrac{1}{k}\right)|M_k|_*,
\]
then letting $\varepsilon\to 0$ we obtain a contradiction.

If  $|S|_*=\infty$, we can write
$S=\cup_{j=1}^\infty S_j$ with $|S_j|_*<\infty$. Let 
$f_j(x)=\limsup_{r\to 0} \dfrac{|S_j\cap B_r(x)|_*}{|B_r(x)|}$. We have $f_j\leq f$, and so
\begin{align*}
\{x\in S:f(x)<1\}= \cup_{j=1}^\infty \{x\in S_j:f(x)<1\}
\subseteq  \cup_{j=1}^\infty \{x\in S_j:f_j(x)<1\}
\end{align*} 
and the lemma follows.

This argument applies to the case of a general metric space, because the Vitali covering theorem in that context is available from the book of Ambrosio and Tilli \cite[Theorem 2.2.2.]{ambrosio-tilli:bookmetricspaces},
and then the second part of our lemma follows.

\end{proof}

As a consequence of Lemma \ref{lem:aleksandrovforreflectors} we obtain the following.
\begin{prop}\label{propmeasu}
Suppose $\sigma$ is a reflector from $\bar{\Omega}$ to $D$ and the target $D$ is contained in a plane or $D$ is countable.
If $A$ and $B$ are disjoint subsets of $D$, then $\tau_\sigma(A) \cap \tau_\sigma(B)$ has Lebesgue measure zero.
\end{prop}

\begin{remark}\rm
We assume that the target $D$ is compact, and let us set $M=\max_{P\in D}OP$.
We are interested in reflectors $\sigma$ such that they are at a positive distance from the origin.
Say we want $\rho(x)\geq a>0$ for all $x \in \Omega$.  To obtain this, by Proposition \ref{propineq}  it is enough to pick $\delta>0$ such that $\dfrac{\delta\,M}{1+c_{\delta}}\geq a$ and impose the condition that $d\geq \delta M$ for each supporting ellipsoid $E_d(P)$ to $\sigma$.
\end{remark}
\begin{defi}\label{def:A(delta)}
Suppose $D$ is compact not containing $O$, and $M=\max_{P \in D} OP$.
For each $\delta>0$, we let
$
\mathcal{A(\delta)}$ be the collection of all reflectors $\sigma=\{\rho(x)x\}_{x\in \bar{\Omega}}$ from $\bar{\Omega}$ to $D$, such that for each $x_0 \in \bar{\Omega} $ there exist a supporting ellipsoid $E_d(P)$ to $\sigma$ at $\rho(x_0)\,x_0$ with $P \in D$ and $d\geq \delta\, M$.
\end{defi}

The following proposition will be used in Section \ref{sec:solutionforgeneralmeasure}.
\begin{prop}\label{lipschitz}
If $\sigma=\{\rho(x)x\}_{x\in \bar{\Omega}}$ is a reflector in $\mathcal{A(\delta)}$, then $\rho$ is globally Lipschitz in $\bar{\Omega}$, $\rho$ is bounded\footnote{The bound is not necessarily uniform for all $\sigma\in \mathcal A(\delta)$.}, and the surface $\sigma$ is strictly convex. The Lipschitz constant of $\rho$ is bounded uniformly by a constant depending only on $\delta$ and $M$.
\end{prop}

\begin{proof}
Let $x,y \in \bar{\Omega}$. Then there exist $P \in D$ with $d\geq \delta M$ such that $E_d(P)$ supports $\sigma$ at $\rho(x)x$, i.e.,
$\rho(z) \leq \rho_d(z)$ for all $z \in \bar{\Omega}$ with equality at $z=x$.
Therefore
\begin{align*}
\rho(y) - \rho (x) &= \rho(y)-\rho_d(x)\leq \rho_d(y)-\rho_d(x)= \dfrac{d}{1-\varepsilon \, y\cdot m }-\dfrac{d}{1-\varepsilon \, x\cdot m }\\
&=\dfrac{\varepsilon d}{(1-\varepsilon \, x\cdot m )(1-\varepsilon \,  y\cdot m )} (y-x)\cdot  m\\
&\leq \dfrac{\varepsilon d}{(1-\varepsilon)^2} |x-y|
=\dfrac{(1-\varepsilon^2)\,OP}{2\,(1-\varepsilon)^2}\, |x-y|  \qquad \text {by (\ref{eqop})}\\
&=\dfrac{OP}{2}\, \dfrac{1+\varepsilon}{1-\varepsilon} \,|x-y| \leq \dfrac{M}{2} \,\dfrac{1+c_\delta}{1-c_\delta}\, |x-y| \qquad \text{by \eqref{epsbdd}.}
\end{align*}
 Interchanging the roles of $x,y$, we conclude that
 $|\rho(y) - \rho (x)| \leq \dfrac{M}{2} \,\dfrac{1+c_\delta}{1-c_\delta}\, |x-y|.$

To prove boundedness of $\rho$, from the definition of $\mathcal A(\delta)$ we have $\rho(x) \geq \dfrac{\delta M}{1+c_\delta}$.
Let $x_0 \in \bar{\Omega}$, then there exist $E_{d_0}(P_0)$ supporting $\sigma$ at $\rho(x_0)x_0$. By Proposition \ref{propineq},
$\rho(x) \leq \rho_{d_0}(x) \leq \dfrac{d_0}{1-c_{\delta}}$for every $x
\in \bar{\Omega}$.

Let $\Pi_0$ the plane tangent to $E_{d_0}(P_0)$ at $\rho(x_0)x_0$ and let $p_0(x)$ be it's polar equation. By strict convexity of the ellipsoid, we have
$\rho_{d_0}(x) < p_0(x)$ for all $x \in \bar{\Omega}$, $x\neq x_0$, and $\rho_{d_0}(x_0)=p_0(x_0)$.
Therefore the strict convexity follows.

 \end{proof}

\begin{prop}\label{Harnack}
If $\sigma=\{\rho(x)x\}_{x\in \bar{\Omega}}$ is a reflector in $\mathcal A(\delta)$, then $\sigma$ admits the following "Harnack type" inequality:
$$\max_{x\in \bar{\Omega}}\rho(x) \leq \dfrac{1+c_{\delta}}{1-c_{\delta}} \min_{x\in \bar{\Omega}} \rho(x).$$
\end{prop}

\begin{proof}
Since $\sigma \in \mathcal A(\delta)$, then $\rho$ is bounded below. 
Let $x_0 \in \bar{\Omega}$ be the point where $\rho$ attains its minimum. Moreover , there exist $P_0 \in D$ and $d_0 \geq \delta M$ such that $E_{d_0}(P_0)$ supports $\sigma$ at $\rho(x_0)x_0$.
By Proposition \ref{propineq},
$
\min_{x\in \bar{\Omega}}\rho(x)=\rho(x_0)=\rho_{d_0}(x_0)
\geq \dfrac{d_0}{1+c_{\delta}}$,
and for every $x \in \bar{\Omega}$
$\rho(x) \leq \rho_{d_0}(x) \leq \dfrac{d_0}{1-c_\delta}\leq \dfrac{1+c_\delta}{1-c_{\delta}}\,\min_{x\in \bar{\Omega}}\rho(x).$
\end{proof}

We now continue  proving properties of reflectors that will permit us to define the reflector measure given in Proposition \ref{mumeasure}.

\begin{prop}\label{convprop}
Assume $O\notin D$ with $D$ compact such that either $D$ is contained on a plane or $D$ is countable.
Let $\sigma \in \mathcal{A(\delta)}$ and let $S$ be the set from
Lemma \ref{lem:aleksandrovforreflectors}. Suppose $\{x_n\}_{n=1}^\infty,x_0$ are in $\bar{\Omega} \setminus S$ and $x_n\to x_0$.
If $E_{d_n}(P_n)$ and $E_{d_0}(P_0)$ are the corresponding supporting ellipsoids to $\sigma$ at $x_n$ and $x_0$, and $\nu(x_n),\nu(x_0)$ are the corresponding unit normal vectors, then we have
\begin{enumerate}
\item $\lim_{n \to \infty} d_{n}= d_0$
\item $\lim_{n \to \infty} P_{n}=P_0$
\item $\lim_{n \to \infty} \nu(x_n)=\nu(x_0)$
\end{enumerate}
\end{prop}

\begin{proof}
We prove that for every subsequence $\{n_k\}$ there exists a sub-subsequence $\{n_{k_l}\}$ for which (1)-(3) hold.
Indeed, we have $\lim_{n \to \infty} x_{n_k} = x_0$.
Since $D$ is compact,
there exists a subsequence $P_{n_{k_l}}$ converging to some $P'_0\in D$. Since $OP'_0 >0$, we then have
$m_{n_{k_l}} = \dfrac{\overrightarrow{OP_{n_{k_l}}}}{OP_{n_{k_l}}}\to m'_0 =\dfrac{\overrightarrow{OP'_0}}{OP'_0} $.

Let $M_{n_{k_l}}=\rho(x_{n_{k_l}})x_{n_{k_l}}$, and $M_0=\rho(x_0)x_0$, then by continuity of $\rho$, $M_{n_{k_l}}\to M_0$
as $l\to +\infty$. Then $c_{n_{k_l}}= OM_{n_{k_l}} + M_{n_{k_l}}P_{n_{k_l}}\to c'_0=OM_0+M_0P'_0>0$. 
By definition of eccentricity, $\varepsilon_{n_{k_l}} =\dfrac{OP_{n_{k_l}}}{c_{n_{k_l}}}$, and $\varepsilon_{n_{k_l}}\to \varepsilon'_0=\dfrac{OP'_0}{c'_0}$.
Notice that since $\sigma\in \mathcal A(\delta)$, we have $\varepsilon_{n_{k_l}}\leq {c_{\delta}} <1$,  and so $0<\varepsilon'_0\leq c_{\delta}<1$ because $OP'_0>0$.
In addition, by equation (\ref{eqop}), $d_{n_{k_l}}\to d'_0$ with $OP_0'=\dfrac{2\varepsilon_0'd_0'}{1-\varepsilon_0'^2}$, 
obtaining $d'_0\geq \delta M$.

Let us now consider the ellipsoid $E_{d'_0}(P'_0)$. Note that
$\rho(x) \leq \dfrac{d_{n_{k_l}}}{1-\varepsilon_{n_{k_l}}\, x\cdot  m_{n_{k_l}}}$ for all $x\in \bar{\Omega}$, with equality at $x=x_{n_{k_l}}$.
If $l\to +\infty$, we then obtain that
$\rho(x) \leq \dfrac{d'_0}{1-\varepsilon'_0\, x\cdot m'_0 }$ for all $x\in \bar{\Omega}$, with equality at $x=x_0$.
Therefore $x_0 \in \tau_{\sigma}(P'_0) \cap \tau_{\sigma}(P_0)$, but since $x_0 \in S^c$ we get $P'_0=P_0$.
The ellipsoids $E_{d'_0}(P_0), E_{d_0}(P_0)$ have the same foci and a common point $\rho(x_0)x_0$, then they are identical and
so $d'_0=d_0$, $\varepsilon'_0=\varepsilon_0$, and $m'_0 =m_0$. Therefore, by equation (\ref{nueq}), $\nu(x_{n_{k_l}})\to \nu(x_0)$ as $l\to \infty$.

 Hence the proposition follows.
\end{proof}

\begin{defi}\label{sigmaalgeb}
Let $\delta>0$ and let $D$ be compact with $O\notin D$. Given $\sigma\in\mathcal{A(\delta)}$, we define
$\mathcal{S_\sigma}=\{\ E\subseteq D:\text{$\tau_\sigma (E)$ is Lebesgue measurable} \}$.
\end{defi}

\begin{lem}\label{nonempty}
If $E \subseteq D$ is closed, then  $E \in \mathcal{S_{\sigma}}$.
\end{lem}

\begin{proof}
It's enough to prove that $\tau_{\sigma}(E)$ is a closed subset of $\bar{\Omega}$.
If $\tau_{\sigma}(E) = \emptyset$, the result is trivial.
Let $x_n$ be a sequence in $\tau_{\sigma}(E)$ with $x_n\to x_0$,  and let $E_{d_n} (P_n)$ ellipsoids supporting $\sigma$ at $\rho(x_n)x_n$. 
Since $E$ is compact, we can use the proof of Proposition \ref{convprop} to show that there exist a subsequence of ellipsoids
$E_{d_{n_k}}(P_{n_k})$ converging to some ellipsoid $E_{d_0}(P_0)$ such that $d_0\geq \delta M$, $P_0 \in E$, and $E_{d_0}(P_0)$ supports $\sigma$ at $\rho(x_0)x_0$. Therefore $x_0 \in
\tau_{\sigma}(E)$.
\end{proof}

\begin{prop}\label{propsigma}
If $O\notin D$ with $D$ a compact set  contained either on a plane or $D$ is countable, then
 $\mathcal{S_{\sigma}}$ is a sigma-algebra on D containing all Borel sets.
\end{prop}

\begin{proof}
By Lemma \ref{nonempty},
 $\mathcal{S_{\sigma}}$ is non empty.
That $\mathcal{S_{\sigma}}$ is closed under countable unions is immediate since
$\tau_\sigma \left(\cup_{i=1}^{+\infty} F_i\right)=\cup_{i=1}^{+
\infty} \tau_\sigma (F_i)$.
To show that $\mathcal{S}_\sigma$ is closed under complements, let $E\in \mathcal{S}_\sigma$.
We have
\begin{align*}
\tau_\sigma(E^c)&=\mathcal{N_\sigma} ^{-1}(E^c)\\
&= \{\ x:\mathcal{N_\sigma} (x) \cap E=\emptyset\}\cup \{\ x: \mathcal{N_\sigma} (x) \cap E \neq \emptyset, \mathcal{N_\sigma} (x) \cap E^c \neq \emptyset \}\ \\
&= \left(\tau_\sigma (E)\right)^c \cup \left(\tau_\sigma(E^c) \cap \tau_\sigma (E)\right),
\end{align*}
and the measurability of $\tau_\sigma(E^c)$ follows from Proposition \ref{propmeasu}.


By Proposition \ref{nonempty}, $\mathcal S_\sigma$ contains all closed sets, hence it contains all Borel sets.
\end{proof}

We are now ready to define the notion of reflector measure.

\begin{prop}\label{mumeasure}
Assume the target $D$ is a compact set with $O\notin D$ such that $D$ is contained on a plane or $D$ is countable.
Let $f\in L^1(\bar{\Omega})$ be non-negative, and let $\sigma$ be a reflector in $\mathcal{A(\delta)}$ for some $\delta>0$.
We define
\[
\mu(E)=\int_{\tau_{\sigma}(E)} f(x)\, \dfrac{ x\cdot \nu(x)}{\rho ^2(x)}\,dx
\] for each Borel set $E$.
Then $\mu$ is a finite Borel measure on $D$.
\end{prop}

\begin{proof}
By Proposition \ref{propsigma}, $\tau_{\sigma} (E)$ is Lebesgue measurable for each Borel set $E$.
By Proposition \ref{convprop} the function $x\cdot \nu(x)$ is continuous relative to $\bar{\Omega}\setminus S$. Also since $\sigma \in A(\delta)$ then $\rho$ is continuous and bounded below, it follows that the function $\dfrac{x\cdot \nu(x)}{\rho ^2(x)}$ is continuous relative to $\bar{\Omega}\setminus S$.

Let $x\in \bar{\Omega} \setminus S$,  by \eqref{nueq} and Proposition \ref{propineq}, we have
$x\cdot \nu(x) =\dfrac{1-\varepsilon \, x\cdot m}{|x-\varepsilon m|}\geq \dfrac{1-c_{\delta}}{1+c_\delta}>0$, where
$\varepsilon$ is the eccentricity of the supporting ellipsoid to $\sigma$ at $\rho(x)x$.
From Proposition \ref{propmeasu}, $|S|=0$ and therefore $\mu(E)$ is well defined for each $E$ Borel set and is non negative.
To prove the sigma additivity of $\mu$, let $E_1,E_2,...$ be countable mutually disjoint sequence of Borel sets.
 Then by Proposition \ref{propmeasu}, $\mu(E_i \cap E_j)=0$ for all $i\neq j$,
 and hence
 $\mu\left(\cup_{i=1}^{+\infty} E_i\right) = \sum_{i=1}^{+\infty} \mu(E_i)$.
\end{proof}

To complete the list of properties of the reflector measure, we show the following stability property.
\begin{prop} \label{weaklem}
Suppose $D$ is a compact set with $O\notin D$ and such that $D$ is contained on a plane or $D$ countable, and let $f\in L^1(\bar \Omega)$ be non-negative.
Let $\sigma_n$ be a sequence of reflectors in $\mathcal{A(\delta)}$ for some fixed $\delta>0$, where $\sigma_n=\{ \rho_n (x)\,x \}_{x\in \bar{\Omega}}$ are such that $\rho_n(x)\leq b$ for all $x \in \bar{\Omega}$, for all $n$ and for some $b>0$, and $\rho_n$ converges point-wise to $\rho$ in $\bar{\Omega}$.
Let $\sigma=\{\rho(x)x\}_{x\in \bar{\Omega}}$. Then we have
\begin{enumerate}
\item $\sigma\in \mathcal{A(\delta)}$, i.e., for all $x \in \bar{\Omega}$ there exist $P$ in $D$ and $d \geq \delta M$ such that $E_{d}(P)$ supports $\sigma$ at $\rho(x)x$.
\item If $\mu$ is the reflector measure corresponding to $\sigma$, then $\mu_n$ converges weakly to $\mu$ .
\end{enumerate}
\end{prop}

\begin{proof}
First notice that $\dfrac{\delta M}{1+c_{\delta}}\leq \rho_n(x)\leq b$, and so $\rho$ satisfies the same inequalities.
Let us prove (1). Take $x_0 \in \bar{\Omega}$, then there exist $P_n$ in $D$, and $d_n \geq \delta M$, such that $E_{d_n}(P_n)$ supports $\sigma_n$ at $\rho(x_0)\,x_0$. By Proposition \ref{propineq}, we have
$b\geq \rho_n(x_0)=\rho_{d_n}(x_0)\geq \dfrac{d_n}{1+c_\delta}$,
concluding that
$\delta M \leq  d_n \leq b(1+c_\delta)$.
Therefore, there exist a subsequences $d_{n_k}$, and $P_{n_k}$ converging to $d_0$, and $P_0$, respectively, such that $P_0 \in D$, and $\delta M \leq d_0 \leq b\,(1+c_\delta)$.
Moreover, by equation (\ref{epsequ}) the eccentricity of $E_{d_{n_k}}(P_{n_k})$ converges to the eccentricity $\varepsilon_0$ of $E_{d_0}(P_0)$ and $0<\varepsilon_0\leq c_\delta$, therefore $\rho_{d_{n_k}}$ converges to $\rho_{d_0}$.

On the other hand,
$\rho_{n_k}(x) \leq \rho_{d_{n_k}}(x) $, with equality at $x=x_0$. Letting $k\to +\infty$, we get that $\rho(x) \leq \rho_{d_0} (x)$ with equality at $x=x_0$.
Hence $E_{d_0}(P_0)$ supports $\sigma$ at $\rho(x_0)x_0$ where $P_0 \in D$ and $d_0 \geq \delta M$.
Hence, $\sigma \in \mathcal{A(\delta)}$.

To prove (2), let
$S_n$ and $S$ be the sets in Lemma \ref{lem:aleksandrovforreflectors}
corresponding to the reflectors $\sigma_n$ and $\sigma$, respectively.
We have $|S_n|=|S|=0$, and so
 $H:= \cup_{i=1}^{+\infty} S_n \cup S$ has surface measure zero.
For $x\in \bar{\Omega} \setminus H$, we have that $\mathcal{N}_{\sigma_n}(x)$ and $\mathcal{N_{\sigma}}(x)$ are single valued, and by Proposition \ref{convprop} the normal vectors  $\nu_n(x),\nu(x)$ are continuous relative to $\bar{\Omega}\setminus H$.
Let $h \in C(D)$, we have
$$\int_D h \,d{\mu_n}=\int_{\bar{\Omega} \setminus H} h\left(\mathcal{N}_{\sigma_n}(x)\right) \,f(x)\,\dfrac{x\cdot \nu_n(x)}{\rho_n ^2(x)}\,dx.$$
Since $h$ is continuous on a compact set,  $\rho_n(x)\geq\dfrac{\delta M }{1+c_\delta}$, and $ x\cdot \nu_n(x) \leq 1$, it follows that the integrand is bounded uniformly in $n$ by an integrable function over $\bar{\Omega}\setminus H$. So we can use Lebesgue dominated convergence theorem if the point-wise limit of the integrand exists.

By assumption, $\rho_n\to \rho$ point-wise.
We claim that $\mathcal{N}_{\sigma_n}(x)\to \mathcal{N_{\sigma}}(x)$ for each $x\in \bar{\Omega}\setminus H$.
Indeed, since $\mathcal{N}_{\sigma_n}$ and $\mathcal{N}_{\sigma}$ are single valued on $\bar{\Omega}\setminus H$,
we have $\mathcal{N}_{\sigma_n}(x)=\{P_n\}$ and $\mathcal{N}_{\sigma}(x)=\{P_0\}$ ($P_n$ and $P_0$ unique depending on $x$).
Let $P_{n_k}$ be a subsequence of $P_n$, then
from the proof of (1) above, $P_{n_k}$ has a convergent subsequence to $\bar P$ where $\bar P \in
\mathcal{N_{\sigma}}(x)$. Since $\mathcal{N}_{\sigma}$ is single valued at $x$ we get $\bar P=P_0$
and the claim is proved.
Since $h$ is continuous, we conclude that $\lim_{n \to +\infty} h(\mathcal{N}_{\sigma_n}(x))=h(\mathcal{N_{\sigma}}(x))$ for all $x \in \bar{\Omega} \setminus H$.

For each $x\in \bar{\Omega}\setminus H$, there exist a unique sequence of ellipsoids $E_{d_n}(P_n)$ that supporting $\sigma_n$ at $\rho_n(x)x$, and these ellipsoids
converge to $E_{d_0}(P_0)$ which support $\sigma$ at $\rho(x)x$.
We show that the normals $\nu_n (x)$ converge to $\nu(x)$ for $x\in \bar{\Omega}\setminus H$.
In fact,
from equation (\ref{nueq}) we only need to show that $\varepsilon_n\to \varepsilon_0$ and $m_n\to m_0$.
That $m_n\to m_0$ is a consequence that $P_n\to P_0$.
By definition of eccentricity, $\varepsilon_n=\dfrac{OP_n}{c_n}$ with $c_n=OM_n+M_nP_n$ where $M_n=\rho_n(x)x$.
Since $\rho_n(x)\to \rho(x)$ we get that $c_n\to c_0=OM_0+M_0P_0$ with $M_0=\rho(x)x$.

 We conclude that
\begin{equation*}
\lim_{n \to +\infty}\int_{\bar \Omega \setminus H} h(\mathcal{N}_{\sigma_n}(x))\,f(x)\, \dfrac{ x\cdot \nu_n(x)}{\rho_n ^2(x)}dx=\int_{\bar \Omega \setminus H} h(\mathcal{N_{\sigma}}(x)) \,f(x)\,\dfrac{ x\cdot \nu(x) }{\rho^2(x)}dx,
\end{equation*}
and so
\begin{equation*}
\lim_{n \to \infty} \int_D h \,d\mu_n = \int_D h \,d\mu \qquad \text{for each $h \in C(D)$.}
\end{equation*}
\end{proof}

\begin{cor}\label{continui}
If $D=\{P_1,P_2,...,P_N\}$ and $\sigma_n,\sigma,\mu_n,\mu$ are as in Proposition \ref{weaklem}, then
$$\lim_{n \to +\infty} \mu_n(P_i)=\mu(P_i).$$
\end{cor}

\begin{proof}
Define $h_i(P_j)=\delta_{i}^{j}$, $1\leq i,j\leq N$. Since $D$ is discrete, $h$ is continuous on $D$, then by the previous proposition
$$\lim_{n \to +\infty} \mu_n(P_i)=\lim_{n \to +\infty} \int_D h_i(y)d\mu_n(y)=\int_D h_i(y)d\mu(y)=\mu(P_i).$$
\end{proof}

\subsection{Physical visibility issues}\label{subset:physicalvisibilityissues}
With the Definition \ref{def:definitionofreflector} of reflector, the reflected rays might cross the reflector to reach the target, in other words, the reflector might obstruct the target in certain directions. This is illustrated in Figure \ref{fig:focioutside}.
In this Subsection, we show by convexity that if the ellipsoids used in the definition of reflector are chosen such that they contain all points of $D$ in their interiors, this obstruction can be avoided, that is, the reflector will not obstruct the target in any direction. For example, in Figure \ref{fig:fociinside} each reflected ray will not cross the reflector to reach the target.

Indeed, let $\{E_{d_i}\}_{i\in I}$ be a family of ellipsoids with foci $O$ and $P_i$, such that  the convex body $B$ enclosed by all $\{E_{d_i}\}_{i\in I}$ is a reflector.
Let us assume that all $P_i$'s are in the interior of $B$, and $D=\{P_i\}_{i\in I}$ is compact.
We shall prove that under this condition any ray emanating from $O$ is reflected into a ray that does not cross the boundary of $B$ to reach the target. 
\begin{figure}[htp]
\begin{center}
\subfigure[$P=(1,0)$, $Q=(0,1)$, $E_{1.2}(P)$ and $E_{1.4}(Q)$]{\label{fig:focioutside}\includegraphics[width=3in]{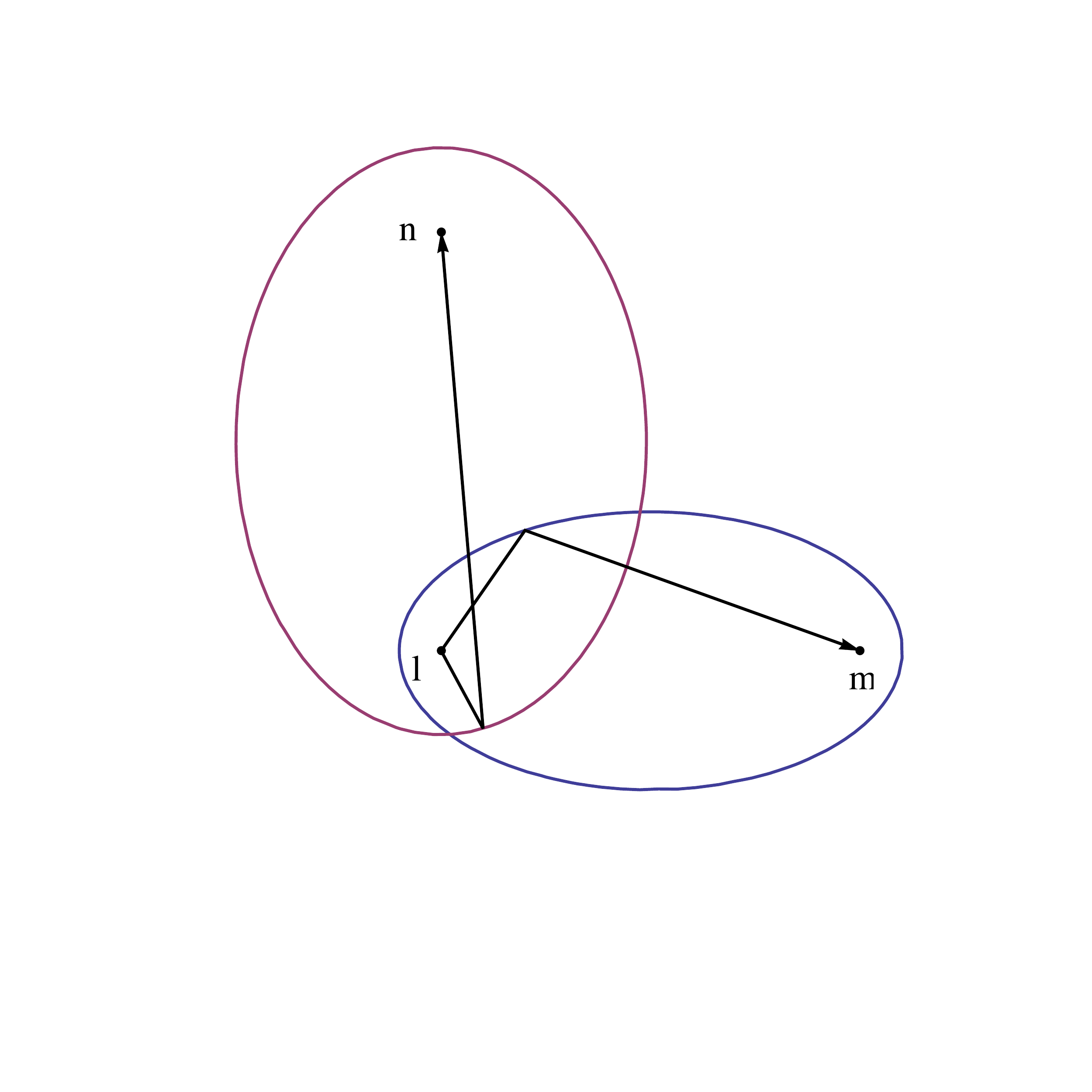}}
    \subfigure[$P=(-1,.2)$, $Q=(1,0)$, $E_{5.5}(P)$ and $E_5(Q)$]{\label{fig:fociinside}\includegraphics[width=3in]{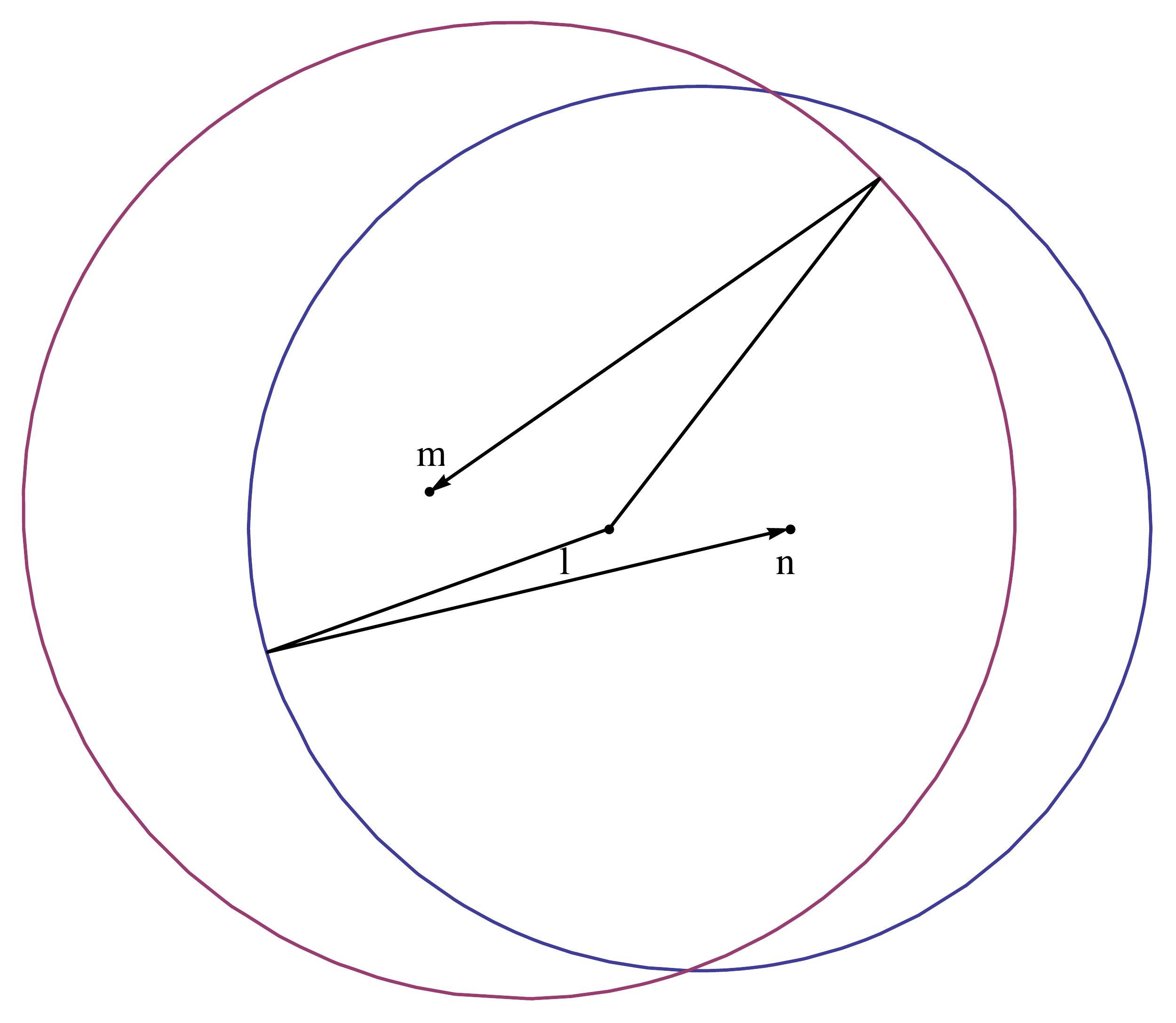}}
\end{center}
  \caption{Reflectors illustrating obstruction}
  \label{fig:obstruction}
\end{figure}
Suppose by contradiction that there is ray $r$ emanating from $O$ so that the reflected ray $r'$ crosses the boundary of $B$ to reach the target.
Then $r$ hits $\partial B$ at some point $P$,
and so $P$ is on the boundary of some ellipsoid $E_{d_i}$, and the reflected ray $r'$ crosses the boundary of $B$ at a point $Q$ to reach the target at say $P_i$.
Since $P_i$ is in the interior of $B$, $P\in \partial B$, and $B$ is convex, the segment
$(1-t)P+tP_i \in \text{Int }(B)$  for all $0<t\leq 1$. Since for some $t$, $Q$ is on this segment, then $Q$
belongs to the interior of $B$, a contradiction.



To assure that each supporting ellipsoid in the definition of reflector contain all points in $D$, we proceed as follows.
Take  $m=\min_{P\in D} OP, M=\max_{P\in D} OP$.
Let $P_i,P_j \in D$, $P_i$ is inside the body of the ellipsoid with focus $O$ and $P_j$ if and only if
\begin{align*}
OP_i+P_iP_j &<c_j=\dfrac{OP_j}{\varepsilon_j},
\end{align*}
that is, if and only if
$\varepsilon_j< \dfrac{OP_j}{OP_i+P_iP_j}$ for all $P_i,P_j \in D$.
Therefore, in the definition of reflector it is enough to choose ellipsoids with eccentricity $\varepsilon$ satisfying
\begin{equation}\label{eps}
\varepsilon < \dfrac{m}{M+\diam(D)}.
\end{equation}
By equation \eqref{eqop},
$d_j=\dfrac{(1-\varepsilon_j^2)\, OP_j}{2\,\varepsilon_j}$. Hence by monotonicity of $d$ we have that (\ref{eps}) is then equivalent to
$$d_j> \dfrac{1-\left(\dfrac{m}{M+\diam(D)}\right)^2}{2\dfrac{m}{M+\diam(D)}} OP_j \qquad \text{for all $P_j \in D$},$$
and so it is enough to choose
$d> \dfrac{1-\left(\dfrac{m}{M+\diam(D)}\right)^2}{2\,\dfrac{m}{M+\diam(D)}}\, M$.

Therefore, to avoid obstruction of the target by the reflector, we can consider reflectors in the class $\mathcal A(\delta)$ with
\begin{equation}\label{eq:assumptionondeltatoavoidobstruction}
\delta>\dfrac{1-\left(\dfrac{m}{M+\diam(D)}\right)^2}{2\,\dfrac{m}{M+\diam(D)}}:=\delta_D.
\end{equation}

To complete this Subsection, we mention the case when the target is on the way to the reflector,
that  is, the incident rays cross the target before reaching the reflector.
Clearly, this can be avoided if we assume that $\bar \Omega \cap D^*=\emptyset$, where $D^*$ is the projection of $D$ on $S^2$.
\setcounter{equation}{0}
\section{Solution of the problem in the discrete case}\label{sec:Soldis}

\begin{defi}\label{Reflector}
Let $\Omega\subseteq S^{2}$ with $|\partial \Omega|=0$, $D=\{P_1,P_2,...,P_N\}$ is such that $O\notin D$, and $M=\max_{P\in D}OP$.
Let $d_1,\cdots ,d_N$ be positive numbers and $w=(d_1,d_2,\cdots ,d_n)$.
Define the reflector $\sigma=\{\rho_w(x)x\}$, $x\in \bar{\Omega}$, where
\[
\rho_w(x)=\min_{1\leq i \leq N} \rho_{d_i}(x),
\]
with $\rho_{d_i}(x)=\dfrac{d_i}{1-\varepsilon_i\, x\cdot m_i}$, $\varepsilon_i=\sqrt{1+\dfrac{d_i^2}{OP_i^2}}-\dfrac{d_i}{OP_i}$ and $m_i=\dfrac{\overrightarrow{OP_i}}{OP_i}$, $OP_i=|\overrightarrow{OP_i}|$.
\end{defi}


\begin{lem}\label{blwbdd}
Let $0<\delta\leq \delta '$, and let $\{\rho_w(x)x\}$ be the reflector with $w=(d_1,\cdots ,d_N)$, where $d_1\leq \delta'\,M$ and $d_i \geq \delta\,M$ for $1\leq i \leq N$.
If $f\in L^1(\bar{\Omega})$ and $f>0$ a.e
then
\begin{equation}\label{cond}
\mu_w(D)=\int_{\bar{\Omega}}f(x)\,\dfrac{ x\cdot \nu_w(x)}{\rho_w ^2 (x)}\,dx > C(\delta,\delta ',M)
\int_{\bar{\Omega}} f(x)\,dx
\end{equation}
where $C(\delta,\delta',M)$ is a constant depending only on $\delta$, $\delta'$ and $M$.
\end{lem}

\begin{proof}
From inequality (\ref{impineq}), $\rho_{d_1}(x) \leq \dfrac{d_1}{1-c_\delta} \leq \dfrac{\delta'\,M}{1-c_\delta}$,
and so by the definition of reflector $\rho_w(x)\leq \dfrac{\delta'\,M}{1-c_\delta}$.
From Proposition \ref{lipschitz}, the set of singular points of the reflector $\rho_w$ has measure zero.
So for each $x \in \Omega$ not a singular point, there exists $1\leq i \leq N$ such that
\begin{equation}\label{eq:inequalityatnonsingularpoints}
 x\cdot \nu_w(x) = x\cdot \nu_{d_i}(x)
= \dfrac{1-\varepsilon_i \, x\cdot m_i}{|x-\varepsilon_i m_i|}\geq \dfrac{1-\varepsilon_i}{1+\varepsilon_i}
\geq \dfrac{1-c_{\delta}}{1+c_{\delta}}
\end{equation}
by Propositions \ref{propineq} and \ref{propnorm}.
Therefore $\dfrac{x \cdot \nu_w(x)}{\rho_w ^2 (x)}\geq \dfrac{(1-c_\delta)^3}{(1+c_{\delta})\,(\delta'\,M)^2}:=C(\delta,\delta ',M)$ for a.e. $x$, and consequently
$$\mu_w(D)=\int_{\bar{\Omega}}f(x)\,\dfrac{ x\cdot \nu_w(x)}{\rho_w ^2 (x)}\,dx \geq C(\delta,\delta ',M)\int_{\bar{\Omega}} f(x)\,dx.$$
To prove the strict inequality, suppose by contradiction that we have equality.
Since $f>0$ a.e., we then would get that
\begin{equation}\label{eq:ratioequalconstant}
\dfrac{ x\cdot \nu_w(x)}{\rho_w ^2 (x)} = C(\delta,\delta ',M),\quad \text{for a.e. $x\in \Omega$}.
\end{equation}
If $x_0$ is not a singular point of $\rho_w$ and $x_0\notin S$ ($S$ defined in Lemma \ref{lem:aleksandrovforreflectors}), then there exists a unique $1\leq j\leq N$ such that
$\rho_w(x_0)=\rho_{d_j}(x_0)$ and $\rho_w(x_0)<\min_{i\neq j}\rho_{d_i}(x_0)$.
By continuity there exists a neighborhood $V_{x_0}$ of $x_0$ such that
$\rho_w(x)<\min_{i\neq j}\rho_{d_i}(x)$ for all $x\in V_{x_0}$ and therefore $\rho_w(x)=\rho_{d_j}(x)$ for all $x\in V_{x_0}$.
On the other hand, from \eqref{eq:inequalityatnonsingularpoints} and \eqref{eq:ratioequalconstant} we get that
$\rho_w(x)\geq \dfrac{\delta' M}{1-c_\delta}$ for $x$ outside a set of measure zero.
Since $\rho_w$ is continuous, we get that $\rho_w$ is constant on $V_{x_0}$, a contradiction.


%
\end{proof}

The following lemma is similar to \cite[Lemma 9]{kochengin-oliker:nearfieldreflector}.
\begin{lem}\label{compare}
Consider the reflectors $\sigma=\{\rho_w(x)x\}_{x\in \bar{\Omega}}$ and $\tilde{\sigma}=\{\rho_{\tilde w}(x)x\}_{x\in \bar{\Omega}}$, with
$w=(d_1,d_2,\cdots,d_l,\cdots,d_N)$ and
$\tilde{w}=(d_1,d_2,\cdots,\tilde{d_l},\cdots,d_{N})$, such that $\tilde{d_l} \leq d_l$. We write in this case $w\geq _l \tilde{w}$.

 If $\mu$ and $\tilde{\mu}$ are the corresponding reflector measures, then $\tilde{\mu}(P_i) \leq \mu(P_i)$ for $i \neq l$.
\end{lem}
\begin{proof}



Let $1 \leq i \leq N$ with $i \neq l $. If $x \in \tau_{\tilde{\sigma}}(P_i)$, then
$$\rho_{\tilde{d_i}}(x) \leq \rho_{\tilde{d_j}}(x) \qquad \text{ for all $1\leq j \leq N$ } .$$
For $j \neq l$, $\tilde{d_j}=d_j$, so $\rho_{\tilde{d_j}}(x)=\rho_{d_j}(x)$. For $j=l$, since $\tilde{d_l}\leq d_l$, from Proposition
\ref{propmonot} we have $\rho_{\tilde{d_l}}(x) \leq \rho_{d_l}(x)$. Then
$$\rho_{d_i}(x)=\rho_{\tilde{d_i}}(x) \leq \rho_{d_j}(x) \qquad \text{for all $1 \leq j \leq N$. }$$
Therefore, $\tau_{\tilde{\sigma}}(P_i) \subseteq \tau_{\sigma}(P_i)$ for $i\neq l$, and so
\begin{align*}
\tilde{\mu}(P_i)&=\int_{\tau_{\tilde{\sigma}}(P_i)} f(x)\, \dfrac{ x\cdot \tilde{\nu}(x)}{\tilde{\rho}^2(x)}\,dx
=\int_{\tau_{\tilde{\sigma}}(P_i)} f(x)\, \dfrac{ x\cdot \nu_{\tilde{d_i}}(x)}{\rho_{\tilde{d_i}}^2(x)}\,dx\\
&=\int_{\tau_{\tilde{\sigma}}(P_i)} f(x)\, \dfrac{ x\cdot \nu_{d_i}(x)}{\rho_{d_i}^2(x)}\,dx\\
&\leq \int_{\tau_{\sigma}(P_i)} f(x)\, \dfrac{ x\cdot \nu_{d_i}(x)}{\rho_{d_i}^2(x)}\,dx=\int_{\tau_{\sigma}(P_i)} f(x)\, \dfrac{ x\cdot\nu(x)}{\rho^2(x)}\,dx=
 \mu(P_i).
\end{align*}

\end{proof}

As in \cite{kochengin-oliker:nearfieldreflector} we obtain the following corollary.
\begin{cor}\label{corcomp}
Let $w_1=(d_1^1,d_2^1,\cdots ,d_N^1)$ and
$w_2=(d_1^2,d_2^2,\cdots,d_N^2)$. 
Define
$w=(d_1,d_2,\cdots ,d_N)$ where $d_i=\min(d_i^1,d_i^2)$, we write $w=\min(w_1,w_2)$.
Let $\mu_1,\mu_2,\mu$ be their corresponding reflector measures. Then
$$\mu(P_i)\leq \max(\mu_1(P_i),\mu_2(P_i)) \ for\ all\ 1 \leq i \leq N.$$

\end{cor}

\begin{proof}
Fix $i$ with $1\leq i \leq N$.
Without loss of generality, we may assume that $d_i=d_i^1$, i.e., $d_i^1\leq d_i^2$.
Consider $u_0,u_1,u_2,\cdots,u_N$, and their corresponding reflector measures $\nu_0,\nu_1,\nu_2,\cdots,\nu_N$ as follows:
\begin{align*}
u_0&=(d_1^1,d_2^1,\cdots ,d_N^1)=w_1\\
u_1&=(d_1,d_2^1,d_3^1,\cdots ,d_N^1)\\
u_2&=(d_1,d_2,d_3^1\cdots ,d_N^1)\\
\cdots \\
u_j&=(d_1,\cdots,d_j,d_{j+1}^1,\cdots,d_{N}^1)\qquad \text{for $1\leq j \leq N-1$}\\
\cdots \\
u_N&=(d_1,d_2,d_3,\cdots ,d_N)=w.
\end{align*}
With the notation of Lemma \ref{compare}, we have $u_0\geq _1 u_1\geq _2 \cdots \geq _N u_N$. Since $d_i=d_i^1$, we have $u_{i-1}=u_i$ and hence
$$\nu_{i-1}=\nu_i.$$
If $1\leq j \leq i-1$, then by Lemma \ref{compare}
$$\nu_j(P_i) \leq \nu_{j-1}(P_i),$$
and likewise for $i+1\leq j \leq N$
$$\nu_j(P_i) \leq \nu_{j-1}(P_i).$$
Since $u_N=w$, we get $\nu_N(P_i)=\mu(P_i)$.
Therefore
$$\mu(P_i)=\nu_N(P_i)\leq \cdots \leq \nu_i(P_i)=\nu_{i-1}(P_i)\leq \cdots \leq \nu_0(P_i) = \mu_1(P_i) \leq  \max(\mu_1(P_i),\mu_2(P_i)).$$

If $d_i=d_i^2$, i.e., $d_i^2\leq d_i^1$ then we define $u_0=(d_1^2,d_2^2,\cdots ,d_N^2)=w_2$ and
$u_j=(d_1,\cdots,d_j,d_{j+1}^2,\cdots,d_{N}^2)$ and proceed in a similar way.

Since $i$ was arbitrarily chosen, we obtain the corollary.

\end{proof}

We now prove existence of solutions in the discrete case.

\begin{theorem}\label{discretethmgeq}
Let $\Omega\subseteq S^2$ with $|\partial \Omega|=0$,
$f\in L^1(\bar{\Omega})$ such that $f>0$ a.e,
$g_1,g_2,\cdots ,g_N$ positive numbers with $N > 1$.
Let $D=\{P_1,P_2,....,P_N\}$ such that $O\notin D$, 
and let
 $M=\max_{1\leq i\leq N} OP_i$.
Define the measure $\eta$ on D by $\eta=\sum_{i=1}^{N} g_i \delta_{P_i}$.
Fix $\delta > 0$, 
let $k \geq \dfrac{1+c_\delta}{1-c_\delta}$, where $c_\delta$ is from \eqref{epsbdd}, and suppose that
\begin{equation}\label{probcond}
\int_{\bar \Omega}f(x)\,dx \geq \dfrac{1}{C(\delta,k \delta,M)} \eta(D),
\end{equation}
where $C(\delta,k\delta,M)=\dfrac{(1-c_{\delta})^{3}}{(1+c_{\delta})(k\delta M)^{2}}$.

Then there exists a reflector $\bar{w}=(\bar{d_{1}},\cdots ,\bar{d_{N}})$ in $\mathcal A(\delta)$, i.e., with $\bar d_{i}\geq \delta M$ for $1 \leq i \leq N$, satisfying:
\begin{enumerate}
\item $\bar{\Omega}=\bigcup_{i=1}^{N} \tau_{\bar{\sigma}}(P_{i})$.
\item $ \bar{\mu}(P_i)= g_i$ for $ 2 \leq i \leq N$, where $\bar{\mu}$ is the reflector measure corresponding to $\bar{w}$; and
\item $\bar{\mu}(P_1) > g_1$.
\end{enumerate}

\end{theorem}

\begin{proof}
Consider the set:
\begin{align*}
W&=\{w=(d_1,\cdots, d_N):d_1=k\delta M, d_i \geq \delta M,  \\
& \qquad \mu_w(P_i)=\int_{\tau_{\sigma_w (P_i)}}f(x)\,\dfrac{ x\cdot \nu_{w}(x) }{\rho_{w} ^2 (x)}\,dx\leq g_i,i=2,...,N \}.
\end{align*}

We first show that $W\neq \emptyset$. In fact, take $w_{0}=(d_{1}^{0},\cdots ,d_{N}^{0})$ with $d_{1}^{0}=k\delta M$, $d_{i}^{0}=t \delta M$ for $2\leq i \leq N$ with $t \geq1$ to be chosen. By Proposition \ref{propineq}, we have $\rho_{d_{1}^{0}}(x) \leq \dfrac{k\delta M}{1-c_{\delta}}$;
and $ \rho_{d_{i}^{0}}(x) \geq \dfrac{t\delta M}{1+c_{\delta}}$ for  $i=2,\cdots ,N$.
If we pick $t$ sufficiently large, then
$\rho_{d_{1}^{0}}(x)\leq \rho_{d_{i}^{0}}(x)$ for $i=2,\cdots ,N$ and all $x \in \bar{\Omega}$. Hence $\rho_{w_0}(x)=\rho_{d_{1}^{0}} (x)$ for  every $x \in \bar{\Omega}$.
Therefore $\mu_{0}(P_i)=0 < g_i$ for $ i=2,\cdots,N$ and so $(k\delta M,t\delta M,...,t\delta M)\in W$ for $t> k(1+c_{\delta})/(1-c_{\delta}).$

$W$ is closed. In fact, let $w_n=(d_1^n,\cdots,d_N^n)\in W$ converging to $w=(d_1,\cdots,d_N)$, and let $\mu_n$ and $\mu$ be their corresponding reflector
measures. We have $d_1=k\delta M$ and $d_i \geq \delta M$ for $ i=2,\cdots,N$.
We have that $w_n\in \mathcal A(\delta)$, and by Proposition \ref{propineq}: $\rho_{w_n}(x) \leq \dfrac{k\delta M}{1-c_\delta}$.
Then by Corollary \ref{continui}
$$\mu(P_i) = \lim_{n \to \infty}  \mu_n(P_i) \leq g_i \qquad \text{for all $i=2,\cdots,N$}.$$
Therefore $w \in W$.

We next prove that if $w \in W$, then $\mu_w(P_1)>g_1$. In fact, we have
\begin{align*}
\mu_w(P_1)-g_1 &=\left[\mu_w(D) -\left(\mu_w(P_2)+\cdots +\mu_w(P_N)\right)\right]- g_1 \\
&=\mu_w(D)-\left(g_1+\mu_w(P_2)+\cdots +\mu_w(P_N)\right)\\
&\geq \mu_w(D) -(g_1+g_2+\cdots +g_N)\\
&> C(\delta,k \delta,M)\int_{\bar{\Omega}} f(x)\,dx -\eta(D) \qquad \text{by Lemma \ref{blwbdd}}\\
&\geq C(\delta,k \delta, M)\left(\dfrac{1}{C(\delta,k \delta,M)}\eta(D)\right) -\eta(D)=0 \qquad \text{by condition (\ref{probcond})}.
\end{align*}

Let $\bar{d_1}=k \delta M$, and $\bar{d_i}=\inf_{w \in W} d_i$ for $2\leq i\leq N$. Take the reflector $\bar{\sigma}=\{\rho_{\bar{w}} (x)x\}$ and it's corresponding measure $\bar{\mu}$, with
$\bar w=(\bar{d_1},\cdots,\bar{d_N})$. We have that  $\bar{d_i}\geq \delta M$ for $2 \leq i \leq N$. Since $W$ is closed and the $d_i 's$ are bounded below,
the infimum is attained at some reflector $\bar w_i=(k\delta M, \bar d_2^i,\cdots,\bar d_{i-1}^i,\bar{d_i}, \bar d_{i+1}^i,\cdots,\bar d_N^i)\in W$ for $2\leq i \leq N$. Let $\bar \mu_i$
be the reflector measure corresponding to $\bar w_i$. Since $\bar w=\min_{2\leq i\leq N} \bar w_i$,
it follows from Corollary \ref{corcomp} that
$$\bar{\mu}(P_i) \leq \max\left(\bar \mu_2 (P_i),\bar \mu_3 (P_i),\cdots,\bar \mu_N (P_i)\right) \leq g_i \qquad \text{for $2 \leq i \leq N$},$$
and so $\bar{w} \in W$.

It remains to prove that in fact we have   $\bar{\mu}(P_i) = g_i$ for all $i \geq 2$.
Without loss of generality, suppose that the inequality is strict for $i=2$, that is, $\bar{\mu}(P_2) <
g_2$. Take $0 <\lambda<1$, $w_{\lambda}=(k\delta M, \lambda \bar{d_2},\bar{d_3},...\bar{d_N})$, and let $\mu_{\lambda}$ be the corresponding reflector measure.
We claim that $\bar{d_2} > \delta M$. Suppose by contradiction that $\bar{d_2} = \delta M$. Then by Proposition \ref{propineq}, $\rho_{\bar{d_2}} \leq \dfrac{\delta M}{1-c_{\delta}}$ and
$\rho_{\bar{d_1}}\geq \dfrac{ k\delta M}{1+c_{\delta}}$, but since $k \geq \dfrac{1+c_{\delta}}{1-c_{\delta}}$, we have $\rho_{\bar {d_{1}}}  \geq \rho_{\bar{d_{2}}}$.
Therefore
$\tau_{\bar{\sigma}}(P_1) \subseteq \tau_{\bar{\sigma}}(P_2)$, and by Proposition \ref{propmeasu}, we have  $|\tau_{\bar{\sigma}}(P_1)| = |\tau_{\bar{\sigma}}(P_1) \cap
\tau_{\bar{\sigma}}(P_2)|=0$, hence $\bar{\mu} (P_{1})=0$, a contradiction.
This proves the claim, and therefore $\lambda \bar{d_{2}}>\delta M$ for all $\lambda$ sufficiently close to one.
Moreover, by Lemma \ref{compare}, $\mu_{\lambda}(P_{i}) \leq \bar{\mu}(P_{i})\leq g_{i}$  for $i  \geq 3$, and by Corollary \ref{continui}
$$\lim_{\lambda \to 1} \mu_{\lambda}(P_2)= \bar{\mu}(P_2).$$
Then there exist $\lambda_0$ close to one such that $\mu_{\lambda}(P_2) < g_2$ and $\lambda \bar{d_{2}}\geq \delta M$,
for $ \lambda_0 \leq \lambda < 1$. Hence $w_{\lambda} \in W $
contradicting the definition of $\bar{d_2}$.
We conclude that $\bar{w}$ satisfies conditions (1)--(3).
\end{proof}

\begin{remark}\label{rmk:discretesolutionwithoutobstruction}\rm
The reflector constructed in Theorem \ref{discretethmgeq} might obstruct the reflected rays before reaching the target and also can be in the way to the target.
To avoid this, and as it was explained in Subsection \ref{subset:physicalvisibilityissues}, 
if is enough to choose $\delta>\delta_D$, with $\delta_D$ defined in \eqref{eq:assumptionondeltatoavoidobstruction}, and that $\bar \Omega\cap D^*=\emptyset$ where $D^*$ is the projection of $D$ over $S^2$.
\end{remark}

\subsection{Comparison with the solution in \cite[Theorem 7]{kochengin-oliker:nearfieldreflector}}\label{subset:comparisonwithOlikersolution}\rm
The purpose of this subsection is to compare the reflector constructed in \cite[Theorem 7]{kochengin-oliker:nearfieldreflector} where the inverse square law is not taken into account, with the reflector constructed in Theorem \ref{discretethmgeq}.

Let us assume that 
\begin{equation}\label{Olikercondition}
\int_{\bar \Omega} f(x)\,dx=
\dfrac{1}{C(\delta,k\delta,M)}\,\eta(D),
\end{equation}
with $k \geq \dfrac{1+c_\delta}{1-c_\delta}$.
The reflector in \cite[Theorem 7]{kochengin-oliker:nearfieldreflector} is constructed when $\delta=4, k=4$, and the energy measure at the target $D$ equals $\eta(D)/C(4,16,M)$.
With the same method there, it can be proved that given $\delta>0$ and $k \geq \dfrac{1+c_\delta}{1-c_\delta}$, and the densities satisfying \eqref{Olikercondition},
 there exists a solution $\sigma^*=(k\delta M, d_2^*,\cdots,d_N^*) \in \mathcal A(\delta)$ such that
\begin{enumerate}
\item $\bigcup_{i=1}^N \tau_{\sigma^*}(P_i)=\bar{\Omega}$,
\item $\int_{\tau_{\sigma^*} (P_i)}f(x)\,dx=\dfrac{1}{C(\delta,k\delta,M)}\,g_i \quad \forall 1\leq i \leq N$.
\end{enumerate}

Suppose $k=\dfrac{1+c_\delta}{1-c_\delta}$.
We compare $\sigma^*$ with the reflector $\sigma$ constructed in Theorem \ref{discretethmgeq}.
If $\tau_{\sigma^*}(E)$ has positive  measure, then from the argument in
Lemma \ref{blwbdd} we have that
\begin{align*}
\mu^*(E):=\int_{\tau_{\sigma^*}(E)}f(x)\,\dfrac{x\cdot \nu^*(x)}{\rho^*(x)^2}\,dx
&>
C(\delta,k\delta,M)
\int_{\tau_{\sigma^*}(E)}f(x)\,dx
=
\eta(E).
\end{align*}
If $P_1\notin E$, then from
Theorem \ref{discretethmgeq}
\[
\eta(E)=\int_{\tau_{\sigma}(E)}f(x)\,\dfrac{x\cdot \nu(x)}{\rho(x)^2}\,dx.
\]
On the other hand,
if $\rho^*$ is the parametrization of $\sigma^*$, then
\[
\rho^*(x)\geq \dfrac{\delta M}{1+c_\delta},
\]
and obviously
$
x\cdot \nu^*(x)\leq 1$.
Then for each Borel $E\subseteq D$ we have
\begin{align*}
\int_{\tau_{\sigma^*}(E)}f(x)\,\dfrac{x\cdot \nu^*(x)}{\rho^*(x)^2}\,dx
&\leq
\dfrac{(1+c_\delta)^2}{(\delta M)^2}
\int_{\tau_{\sigma^*}(E)}f(x)\,dx\\
&=
\dfrac{(1+c_\delta)^2}{(\delta M)^2}
\dfrac{1}{C(\delta,k\delta,M)}\eta(E)\\
&=
\dfrac{(1+c_\delta)^2}{(\delta M)^2}
\dfrac{(1+c_\delta)(k\delta M)^2}{(1-c_\delta)^3}\eta(E)\\
&=\left(\dfrac{1+c_\delta}{1-c_\delta}\right)^5\eta(E)\\
&\leq
\left(\dfrac{1+c_\delta}{1-c_\delta}\right)^5
\int_{\tau_{\sigma}(E)}f(x)\dfrac{x\cdot \nu(x)}{\rho(x)^2}\,dx.
\end{align*}
Therefore we obtain the estimates
\[
\left(\dfrac{1-c_\delta}{1+c_\delta}\right)^5\int_{\tau_{\sigma^*}(E)}f(x)\,\dfrac{x\cdot \nu^*(x)}{\rho^*(x)^2}\,dx
\leq
\int_{\tau_{\sigma}(E)}f(x)\,\dfrac{x\cdot \nu(x)}{\rho(x)^2}\,dx
<
\int_{\tau_{\sigma^*}(E)}f(x)\,\dfrac{x\cdot \nu^*(x)}{\rho^*(x)^2}\,dx
\]
for each Borel set $E$ with $P_1\notin E$ (the left inequality is true for every $E$ Borel),
As a consequence 
$\sigma^*$ is not a solution in the sense of Theorem \ref{discretethmgeq}, but
for each Borel set $E$ not containing $P_1$ we have
$$0<\mu^*(E)-\eta(E)\leq \left(\left(\dfrac{1+c_\delta}{1-c_\delta}\right)^5-1\right) \eta(D).$$
The error then goes to zero uniformly in $E$ as $\delta\to \infty$.

\subsection{Discussion about overshooting in the discrete case}\label{sub:discreteovershooting}
Theorem \ref{discretethmgeq} shows the existence of a solution that overshoots energy at $P_1$.
We are next interested in finding a solution that minimizes the overshooting at $P_1$.
\begin{defi}\label{classes}
Let $\delta>0$ and $k\geq \dfrac{1+c_\delta}{1-c_\delta}(\geq 1)$. 
With the notation of Theorem \ref{discretethmgeq}
we define the following
\begin{enumerate}
\item $W=\{w=(d_1,\cdots, d_N):d_1=k\delta M, d_i \geq \delta M, \mu_w(P_i) \leq g_i \ \ for \ \ 2\leq i \leq N\}$.
\item The reflector $\bar{\sigma}=\{\rho_{\bar{w}} (x)x\}$, and its corresponding reflector measure $\bar{\mu}$, where $\bar w=(\bar{d_1},\cdots,\bar{d_N})$ with $\bar{d_1}=k \delta M$, and $\bar{d_i}=\inf_{w \in W} d_i$ for $2\leq i \leq N$.
\item $\mathcal{C}= \{w=(d_1,\cdots, d_N): d_1=k\delta M, d_i \geq \delta M, \mu_w(P_i)=g_i \ \ for\ \ 2\leq i \leq N\}$.
\item $\mathcal {D}=\{w=(d_1,\cdots, d_N): d_i \geq \delta M \ \ for\ \  all\ \  i, \mu_w(P_1) \geq g_1,  \mu_w(P_i)=g_i \ \ for \ \ 2\leq i \leq N\}. $
\end{enumerate}
\end{defi}

\begin{prop}\label{classproperties}
Under the assumptions of Theorem \ref{discretethmgeq} we have that
\begin{enumerate}
\item $\bar{w} \in W \cap \mathcal{C} \cap \mathcal{D}$, and hence the three sets are non empty.
\item $\mathcal{C}= W \cap \mathcal{D}$.
\item  $W$ is closed and unbounded, $\mathcal{C}$ and $\mathcal{D}$ are compact.
\end{enumerate}
\end{prop}

\begin{proof}
(1) follows from Theorem \ref{discretethmgeq}.

(2) follows from the fact proved in Theorem \ref{discretethmgeq} that for each $w \in W$, $\mu_w(P_1)>g_1.$

The first part of (3) follows from the proof of Theorem \ref{discretethmgeq}.

To prove that  $\mathcal{D}$ is bounded, let  $w=(d_1,\cdots,d_N)\in \mathcal{D}$. For $1 \leq i \leq N$ we have
\begin{align*}
g_i &\leq \mu_w(P_i)
=\int_{\tau_{\sigma}(P_i)} f(x)\, \dfrac{  x\cdot \nu_w (x)}{\rho_w(x)^2}\,dx
=\int_{\tau_{\sigma}(P_i)} f(x) \dfrac{  x\cdot \nu_{d_i} (x)}{\rho_{d_i}(x)^2}\,dx\\
&\leq \int_{\bar{\Omega}} f(x) \dfrac{1}{\rho_{d_i}(x)^2}\,dx=\int_{\bar{\Omega}} f(x) \dfrac{(1-\varepsilon_i  x\cdot m)^2}{d_i^2}\,dx
\leq \dfrac{4\, \|f\|_{L^1(\bar{\Omega})}}{d_i^2}.
\end{align*}
Therefore
$$d_i \leq 2\sqrt{\dfrac{ ||f||_{L^1(\bar{\Omega})}}{g_i}}$$
so $\mathcal{D}$ is bounded.
To show that $\mathcal{D}$ is closed,
let $w_n=(d_1^n,\cdots,d_N^n)\in \mathcal{D}$ converging to $w=(d_1,\cdots,d_N)$, and let $\mu_n$ and $\mu$ be their corresponding
reflector measures. By Proposition \ref{propineq}, we have:$$\dfrac{\delta M}{1+c_\delta}\leq \rho_{w_n}(x) \leq \dfrac{d_1^n}{1-c_\delta}\leq  \dfrac{2\left( \sqrt{\dfrac{ ||f||_{L^1(\bar{\Omega})}}{g_1}}\right)}{1-c_\delta} .$$
And then by Corollary \ref{continui}, we have
$
\mu(P_i) = \lim_{n \to \infty}  \mu_n(P_i)= g_i$ for $i=2,\cdots,N$, and
$\mu(P_1) =\lim_{n \to \infty}  \mu_n(P_1) \geq g_1$.
Therefore, $w \in \mathcal{D}$ and so $\mathcal{D}$ is closed.

The compactness of $\mathcal{C}$ is hence concluded from (2).

\end{proof}

The following proposition shows that there exist solutions in the sets $\mathcal{C}$ and $\mathcal{D}$ that minimize the overshooting at $P_1$.
\begin{prop}\label{mini}
Under the assumptions of Theorem \ref{discretethmgeq} we have:
\begin{enumerate}
\item There exist $w_{\mathcal{C}} \in \mathcal{C}$ such that $\mu_{w_{\mathcal{C}}}(P_1) \leq \mu_w(P_1)$ for all $w \in \mathcal{C}$.
\item There exist $w_{\mathcal{D}} \in \mathcal{D}$ such that $\mu_{w_{\mathcal{D}}}(P_1) \leq \mu_w(P_1)$ for all $w \in \mathcal{D}$.
\end{enumerate}
\end{prop}

\begin{proof}
We will only prove part $(2)$ because part (1) is similar.
There exists a sequence $w_n \in \mathcal{D}$ such that:
$\lim_{n \to \infty} \mu_n(P_1)= \inf_{w \in \mathcal{D}} \mu_w(P_1)$.
By Proposition \ref{classproperties}, $\mathcal{D}$ is compact, then there exist a subsequence $w_{n_k}$ converging to $w_0\in \mathcal{D}$, and
$\dfrac{\delta M}{1+c_\delta}\leq \rho_{w_{n_k}} 
\leq \dfrac{2}{1-c_\delta}
\sqrt{\dfrac{||f||_{L^1(\bar{\Omega})}}{g_1}}$.
We  conclude from Corollary \ref{continui} that:
$$\mu_{w_0}(P_1) = \lim_{n \to \infty}\mu_{w_{n_k}}(P_1)=\inf_{w \in \mathcal{C}} \mu_w(P_1).$$

\end{proof}


In the following theorem we show that among all solutions in $\mathcal{C}$, the solution that overshoots the minimum amount at $P_1$  is $\bar{w}$. Moreover, if $\bar{\Omega}$ is connected then $\bar{w}$ is the unique such solution.

\begin{theorem}\label{uniqueness}
Let $w=(d_1,d_2,...,d_N)\in \mathcal{C}$ and $\mu$ its corresponding reflector measure. Then $$\bar{\mu}(D) \leq \mu(D),$$
where $\bar \mu$ is the reflector measure corresponding to $\bar w$.
Moreover, if $\bar{\Omega}$ is connected and $\bar{\mu}(D)=\mu(D)$ then:
$$\bar{d_i}=d_i \quad  for\ \ all \ \ 1 \leq i \leq N.$$
\end {theorem}

\begin{proof}
Since $w$ and $\bar{w}$ are in $\mathcal{C}$, then $\mu(P_i)=\bar{\mu}(P_i)$ for all $2 \leq i \leq N$. By definition of $\bar{w}=(\bar d_1,\cdots ,\bar d_N)$ we have
$\bar{d_1} =k \delta M =d_1$ and
$\bar{d_i} \leq d_i$ for all $2 \leq i \leq N$, since $\mathcal{C} \subseteq W$.
Let $\sigma$ and $\bar \sigma$ be the reflectors corresponding to $w$ and $\bar w$, respectively.
Now, let $x \in \tau_{\bar{\sigma}}(P_1)$, i.e., $\rho_{\bar{d_1}}(x) \leq \rho_{\bar{d_i}}(x)$. Then by Proposition \ref{propmonot}:
$$
\rho_{d_1}(x) = \rho_{\bar{d_1}} (x) \leq \rho_{\bar{d_i}} (x) \leq \rho_{d_i} (x) \quad for \ \ all \ \  2\leq i \leq N.
$$
Hence $x \in \tau_{\sigma}(P_1)$ and so $\tau_{\bar{\sigma}}(P_1) \subseteq \tau_{\sigma}(P_1)$.
Therefore:
\begin{align*}
\bar{\mu}(P_1)&= \int_{\tau_{\bar{\sigma}}(P_1)} f(x) \dfrac{ x\cdot\nu_{\bar{d_1}}(x)}{\rho_{\bar{d_1}}^2(x)}\,dx 
= \int_{\tau_{\bar{\sigma}}(P_1)} f(x) \dfrac{ x\cdot \nu_{d_1}(x)}{\rho_{d_1}^2(x)}\,dx,  \qquad \text{since $\bar{d_1}=d_1$}\\
&\leq \int_{\tau_{\sigma}(P_1)} f(x) \dfrac{x\cdot \nu_{d_1}(x)}{\rho_{d_1}^2(x)}\,dx
=\mu(P_1).
\end{align*}
We conclude that $\bar{\mu}(D) \leq \mu(D)$.

Suppose now that $\bar{\Omega}$ is connected and we have equality, i.e.,  $\bar{\mu}(P_i)=\mu(P_i)$ for all $1\leq i\leq N$.
Let $I=\{1\leq i\leq N: \bar{d_i}=d_i\}$ and $J=\{1\leq i\leq N: \bar{d_i} < d_i\}.$ Our goal is to prove that $J$ is empty.
First notice that $I \neq \emptyset$, since $1 \in I$. Similarly as before $\tau_{\bar{\sigma}}(P_i) \subseteq \tau_{\sigma}(P_i)$ for all $i \in I$, and therefore
\begin{align*}
\mu(P_i)&= \int_{\tau_{\sigma}(P_i)} f(x)\, \dfrac{ x\cdot \nu_{d_i}(x)}{\rho_{d_i}^2(x)}\,dx
= \int_{\tau_{\sigma}(P_i)} f(x)\, \dfrac{ x\cdot\nu_{\bar{d_i}}(x)}{\rho_{\bar{d_i}}^2(x)}\,dx, \qquad \text{since $d_i=\bar{d_i}$}\\
&=\int_{\tau_{\bar{\sigma}}(P_i)} f(x)\, \dfrac{ x\cdot \nu_{\bar{d_i}}(x)}{\rho_{\bar{d_i}^2(x)}}\,dx + \int_{\tau_{\sigma}(P_i) \setminus \tau_{\bar{\sigma}}(P_i)} f(x)\, \dfrac{ x\cdot \nu_{\bar{d_i}}(x)}{\rho_{\bar{d_i}}^2(x)}\,dx \\
&= \bar{\mu}(P_i) + \int_{\tau_{\sigma}(P_i) \setminus \tau_{\bar{\sigma}}(P_i)} f(x)\, \dfrac{ x\cdot\nu_{\bar{d_i}}(x)}{\rho_{\bar{d_i}}^2(x)}\,dx.
\end{align*}
Since $\mu(P_i) = \bar{\mu}(P_i)$ and $\dfrac{x\cdot \nu_{\bar{d_i}}(x)}{\rho_{\bar{d_i}}^2(x)}\, f(x) >0$ a.e., we get
$|\tau_{\sigma}(P_i) \setminus \tau_{\bar{\sigma}}(P_i)| = 0$, and so $|\tau_{\sigma}(P_i)| = |\tau_{\bar{\sigma}}(P_i)|$ for $i\in I$.

Suppose now that $J \neq \emptyset$ and let $x \in \bigcup_{j \in J} \tau_{\sigma}(P_j)$, then $x \in \tau_{\sigma}(P_{j_0})$ for some $j_0\in J$. By Proposition \ref{propmonot} we have:
$$ \rho_{\bar w}(x)=\rho_{\bar d_{j_0}} (x) < \rho_{d_{j_0}}(x) \leq \rho_{d_i}(x) = \rho_{\bar{d_i}}(x) \qquad \text{for all $i \in I$}.$$
Then by continuity of $\rho_{\bar w}$,
$$  x \in \text{Int }\left (\bigcup_{j \in J} \tau_{\bar{\sigma}}(P_j)\right) \qquad \text{and so}\ \ \  \bigcup_{j \in J} \tau_{\sigma}(P_j) \subseteq  \text{Int }\left(\bigcup_{j \in J} \tau_{\bar{\sigma}}(P_j)\right).$$
Since $\bar{\Omega}$ is connected and $\bigcup_{j \in J} \tau_{\sigma}(P_j)$ is closed, we get that the set $ A= \bigcup_{j \in J} \tau_{\bar{\sigma}}(P_j) \setminus \bigcup_{j \in J} \tau_{\sigma}(P_j)$ contains the non empty open set
$\left(\text{Int }\left(\bigcup_{j \in J} \tau_{\bar{\sigma}}(P_j)\right)\right) \setminus \bigcup_{j \in J} \tau_{\sigma}(P_j)$ then:
$$ \bigcup_{j \in J} \tau_{\bar{\sigma}}(P_j) = \left (\bigcup_{j \in J} \tau_{\sigma}(P_j)\right) \cup A \qquad \text{with }|A|>0 .$$
This yields a contradiction because
\begin{align*}
|\bar{\Omega}| &= \left|\cup_{k=1}^{N} \tau_{\bar{\sigma}}(P_k)\right|
=\left|\cup_{i\in I} \tau_{\bar{\sigma}}(P_i)\right|+\left|\cup_{j\in J} \tau_{\bar{\sigma}}(P_j)\right|\\
&=\left|\cup_{i\in I} \tau_{\sigma}(P_i)\right|+\left|\cup_{j \in J} \tau_{\sigma}(P_j) \right|+\left|A \right| \\
&=\left|\cup_{k=1}^{k=N} \tau_{\sigma}(P_k)\right|+\left|A\right|
=|\bar{\Omega}| +|A|.
\end{align*}
We then conclude that $J=\emptyset$, and so $d_i =\bar{d_i}$ for all $1\leq i \leq N$.

\end{proof}

\setcounter{equation}{0}
\section{Solution for a general measure $\mu$}\label{sec:solutionforgeneralmeasure}

\begin{theorem}\label{thm:solutionforgeneralmeasure}
Suppose the target $D$ is compact, $O\notin D$, and either $D$ is contained on a plane, or $D$ is countable, and let $M=\max_{P\in D} OP$.
Let $\Omega\subseteq S^2$, with $|\partial \Omega|=0$,  
$f\in L^1(\bar{\Omega})$ with $f>0$ a.e, and let $\eta$ be a Radon measure on $D$.

Given $\delta >0$, $k \geq \dfrac{1+c_\delta}{1-c_\delta}$, 
with $c_\delta$ from \eqref{epsbdd}, we assume that
\begin{equation}\label{probcondgeneral}
\int_{\bar{\Omega}}f(x)\,dx \geq \dfrac{1}{C(\delta,k \delta,M)} \eta(D),
\end{equation}
where $C(\delta,k\delta,M)=\dfrac{(1-c_{\delta})^{3}}{(1+c_{\delta})(k\delta M)^{2}}$.

Given $P_0\in \supp (\eta)$, the support of the measure $\eta$,
there exists a reflector $\sigma=\{\rho(x)x\}_{x\in \bar{\Omega}}$ from $\bar{\Omega}$ to $D$ in $\mathcal A(\delta)$ such that
\[
\eta(E)\leq \int_{\tau_\sigma(E)}f(x)\,\dfrac{ x\cdot \nu_{\rho}(x) }{\rho^2 (x)}\,dx,
\]
for each Borel set $E\subseteq D$, and
\[
\eta(E)= \int_{\tau_\sigma(E)}f(x)\,\dfrac{ x\cdot \nu_{\rho}(x)}{\rho^2 (x)}\,dx,
\]
for each Borel set $E\subseteq D$ with $P_0\notin E$;
and
$\dfrac{\delta\,M}{1+c_\delta}\leq \rho(x)\leq \dfrac{k\delta M}{1-c_\delta}$ for all $x\in \bar \Omega$.
\end{theorem}

\begin{proof}

Partition the domain $D$ into a disjoint finite union of Borel sets with small diameter, say less than $\varepsilon$, so that $P_0$ is in the interior of one of them ($D$ has the relative topology inherited from $\R^3$). Notice that the $\eta$-measure of such a set is positive since $P_0\in \supp (\eta)$. Of all these sets discard the ones that have $\eta$-measure zero.
We then label the remaining sets $D^1_1, \ldots, D^1_{N_1}$ and we may assume $P_0\in (D_{1}^{1})^{\circ}$ and $\eta(D_{j}^{1})>0$ for $1\leq j\leq N_1$.
Next pick $P^1_i \in D^1_i$, so that $P_{1}^{1}=P_{0}$, and
define a measure on $D$ by
\[
\eta_1 = \sum_{i=1}^{N_1} \eta(D^1_i) \delta_{P^1_i}.
\]
Then from \eqref{probcondgeneral}
\begin{equation*}
\eta_1(D) = \eta(D) \leq C(\delta,k \delta,M)\, \int_{\bar \Omega}f(x)\,dx.
\end{equation*}
Thus by Theorem \ref{discretethmgeq}, there exists a reflector
\[
\sigma_1 = \left\{\p_1(x)x: \p_1(x) = \min_{1 \leq i \leq N_1} \dfrac{d_i^1}{1-\varepsilon_i^1\,x\cdot  m^1_i }\right\}
\]
with $d_1^1=k\delta M$, $d_i^1\geq \delta M$ for $2\leq i\leq N_1$, $m^1_i=\dfrac{\overrightarrow{OP^1_i}}{OP^1_i}$ for $1 \leq i \leq N_1$, and satisfying
\begin{equation*}
\eta_1(E) \leq \int_{\tau_{\sigma_1}(E)}f(x)\,\dfrac{ x\cdot \nu_{\rho_1}(x)}{\rho_1^2 (x)}\,dx,
\end{equation*}
with equality if $P_0\notin E$, for each $E$ Borel subset of $D$.

Next subdivide each $D_{j}^{1}$, $1\leq j\leq N_1$, into a finite number of disjoint Borel subsets with diameter less than $\varepsilon/2$, and such that
$P_0$ belongs to the interior of one of the subdivisions of $D_1^1$. Again notice that since $P_0\in \supp (\eta)$, the set in the new subdivision containing $P_0$ has positive $\eta$-measure.
Again discard all sets having $\eta$-measure zero and label them $D^{2}_1, \ldots, D^{2}_{N_{2}}$.
We may assume by relabeling that $D_{1}^{2}\subseteq D_{1}^{1}$ and $P_0\in (D_{1}^{2})^{\circ}$.
Next pick $P^{2}_i\in D^{2}_i$, such that $P_{1}^{2}=P_0$ and consider the measure $\eta_{2}$ on $D$ defined by:
\[
\eta_{2} = \sum_{i=1}^{N_{2}} \eta(D^{2}_i) \delta_{P^{2}_i}.
\]
Then
\begin{equation*}
\eta_{2}(D) = \eta(D) \leq C(\delta,k \delta,M)\,\int_{\bar{\Omega}}f(x)\,d\sigma( x).
\end{equation*}

Once again by Theorem \ref{discretethmgeq}, there exists a reflector
\[
\sigma_{2} = \left\{\p_{2}(x)x: \p_{2}(x) = \min_{1 \leq i \leq N_{2}} \dfrac{d_i^2}{1-\varepsilon_i^2\, x\cdot m^{2}_i }\right\}
\]
with $d_1^2=k\delta M$, $d_i^2\geq \delta M$ for $2\leq i\leq N_2$, $m^2_i=\dfrac{\overrightarrow{OP^2_i}}{OP^2_i}$ for $1 \leq i \leq N_2$, and satisfying
\begin{equation*}
\eta_2(E) \leq \int_{\tau_{\sigma_2}(E)}f(x)\,\dfrac{ x\cdot \nu_{\rho_2}(x) }{\rho_2^2 (x)}\,dx,
\end{equation*}
with equality if $P_0\notin E$, for each $E$ Borel subset of $D$.

By this way for each $\ell =1,2,\cdots $, we obtain a finite disjoint sequence of Borel sets $D_{j}^{\ell}$, $1\leq j\leq N_{\ell}$, with diameters less that $\varepsilon/2^{\ell}$ and $\eta(D_{j}^{\ell})>0$ such that
$P_0\in (D_{1}^{\ell})^{\circ}$, $D_{1}^{\ell+1}\subseteq D_1^{\ell}$, and pick $P_{j}^{\ell}\in D_{j}^{\ell}$ with $P_{1}^{\ell}=P_0$, for all $\ell$ and $j$.
The corresponding measures on $D$ are given by
\[
\eta_\ell = \sum_{i=1}^{N_{\ell}} \eta(D^{\ell}_i) \delta_{P^{\ell}_i}
\]
satisfying
\[
\eta_{\ell}(D) = \eta(D) \leq C(\delta,k \delta,M)\,\int_{\bar \Omega}f(x)\,dx.
\]
We then have a corresponding sequence of reflectors
given by
\begin{equation*}
\sigma_{\ell} = \left\{\p_{\ell}(x)x: \p_{\ell}(x) = \min_{1 \leq i \leq N_{\ell}} \dfrac{d_i^\ell}{1-\varepsilon_i^\ell\, x\cdot m^{\ell}_i }\right\}
\end{equation*}
with $d_1^\ell=k\delta M$, $d_i^\ell \geq \delta M$ for $2\leq i\leq N_\ell$, $m^\ell_i=\dfrac{\overrightarrow{OP^\ell_i}}{OP^\ell_i}$ for $1 \leq i \leq N_\ell$, and satisfying
\begin{equation*}
\eta_\ell(E) \leq \int_{\tau_{\sigma_\ell}(E)}f(x)\,\dfrac{ x\cdot \nu_{\rho_\ell}(x) }{\rho_\ell^2 (x)}dx,
\end{equation*}
with equality if $P_0\notin E$, for each $E$ Borel subset of $D$.
Since $\sigma_\ell\in \mathcal A(\delta)$ for all $\ell$,
it follows by Proposition \ref{lipschitz} that $\rho_\ell$ are Lipschitz continuous in $\bar{\Omega}$
with a constant depending only on $\delta$ and $M$.
In addition, from Proposition \ref{propineq}, and since $d_1^\ell=k\delta M$, we have
\[
\dfrac{\delta\,M}{1+c_\delta}\leq \rho_\ell(x)\leq \dfrac{k\delta M}{1-c_\delta}\qquad \forall \ell, x.
\]
By Arzel\'a-Ascoli theorem, there is a subsequence, denoted also by $\rho_\ell$, converging to $\rho$ uniformly in $\bar \Omega$.
From Proposition \ref{weaklem}, $\sigma=\{\rho(x)x\}$ is a reflector in $\mathcal A(\delta)$ and the reflector measures
$\mu_\ell$, corresponding to $\sigma_\ell$, converge weakly to $\mu$, the reflector measure corresponding to $\sigma$.
We also have that $\eta_\ell$ converges weakly to $\eta$,
and $\eta_\ell(E) = \mu_\ell(E)$ for every Borel set $E\subseteq D$ with $P_0\notin E$, and each $\ell$.
Then  we obtain that
$\eta(E)=\mu(E)$ for every Borel set $E\subseteq D$ with $P_0\notin E$.
Since $\eta_\ell(E) \leq \mu_\ell(E)$ for any Borel set $E \subseteq D$, we also conclude that $\eta(E)\leq \mu(E)$.

\end{proof}

\begin{remark}\label{rmk:generalsolutionwithoutobstruction}\rm
Remark \ref{rmk:discretesolutionwithoutobstruction} also applies to Theorem \ref{thm:solutionforgeneralmeasure}.
\end{remark}

%

\subsection{Discussion about overshooting}\label{sub:overshooting}
In this Subsection, we will discuss the issue of overshooting to the point $P_0\in \supp (\eta)$ and show that there is a reflector that minimizes the overshooting.
Indeed, let $P_0\in \supp (\eta)$ and
\begin{equation}\label{eq:infimumofreflectors}
I=\inf\left\{\int_{\tau_\sigma(P_0)}f(x)\,\dfrac{ x\cdot \nu_{\rho}(x)}{\rho^2 (x)}\,dx:\text{$\sigma$ is a reflector as in Theorem \ref{thm:solutionforgeneralmeasure}} \right\}.
\end{equation}
There exists a sequence of reflectors $\sigma_k=\{\rho_k(x)x\}$ such that
\[
I=\lim_{k\to \infty}\int_{\tau_{\sigma_k}(P_0)}f(x)\,\dfrac{ x\cdot \nu_{\rho_k}(x)}{\rho_k^2 (x)}\,dx.
\]
Therefore from Proposition \ref{lipschitz}, $\rho_k$ are uniformly Lipschitz in $\bar{\Omega}$, and by Theorem \ref{thm:solutionforgeneralmeasure} uniformly bounded. Then by Arzel\'a-Ascoli
there exists a subsequence, also denoted $\rho_k$, converging uniformly to $\rho$.
By Proposition \ref{weaklem}, $\sigma=\{\rho(x)x\} \in \mathcal A(\delta)$, and the corresponding reflector measures $\mu_k$ and $\mu$ satisfy
$\mu_k\to \mu$ weakly.
In particular, $\displaystyle I=\int_{\tau_{\sigma}(P_0)}f(x)\,\dfrac{ x\cdot \nu_{\rho}(x)}{\rho^2 (x)}\,dx$, and we are done.

We now compare $\mu(P_0)$ with $\eta(P_0)$.

\textbf{Case 1:} $\mu(P_0)=\int_{\tau_\sigma(P_0)}f(x)\,\dfrac{ x\cdot \nu_{\rho}(x)}{\rho^2 (x)}\,dx >0$.

In  this case, we shall prove that for each open set $G\subseteq D$, with $P_0\in G$, we have:
\begin{equation}\label{eq:Toveropenisbiggerthanmu}
\int_{\tau_\sigma(G)}f(x)\,\dfrac{x\cdot \nu_{\rho}(x)}{\rho^2 (x)}\,dx
>\eta(G),
\end{equation}
in other words, the reflector overshoots on each open set containing $P_0$.
Notice that from Theorem \ref{thm:solutionforgeneralmeasure} we have equality in \eqref{eq:Toveropenisbiggerthanmu} for each Borel set not containing $P_0$.
Suppose by contradiction there exists an open set $G$, with $P_0\in G$, such that
\[
\int_{\tau_\sigma(G)}f(x)\,\dfrac{ x\cdot \nu_{\rho}(x) }{\rho^2 (x)}\,dx
=\eta(G).
\]
Under this assumption, we are going to prove that $\dfrac{ x\cdot \nu_{\rho}(x)}{\rho^2 (x)}$ is constant a.e.
We have $\tau_\sigma(D)=\bar{\Omega}$, and
$\tau_\sigma(D)=\tau_\sigma(D \setminus G)\cup \tau_\sigma(G)$ where in the union the sets are disjoint a.e.
Then
\begin{align*}
\int_{\bar{\Omega}}f(x)\,\dfrac{x\cdot \nu_{\rho}(x)}{\rho^2 (x)}\,dx
&=
\int_{\tau_\sigma(D)}f(x)\,\dfrac{x\cdot \nu_{\rho}(x)}{\rho^2 (x)}\,dx
\\
&=
\int_{\tau_\sigma(D\setminus G)}f(x)\,\dfrac{x\cdot  \nu_{\rho}(x)}{\rho^2 (x)}\,dx
+
\int_{\tau_\sigma(G)}f(x)\,\dfrac{x\cdot \nu_{\rho}(x)}{\rho^2 (x)}\,dx
\\
&=\eta(D\setminus G)+\eta(G)=\eta(D)
\leq C(\delta,k \delta,M)\,\int_{\bar{\Omega}}f(x)\,dx
\end{align*}
from \eqref{probcondgeneral}, and so we get
\begin{equation}\label{eq:integraltauminusconstantnonpositive}
\int_{\bar{\Omega}}f(x)\,\left(\dfrac{x\cdot \nu_{\rho}(x)}{\rho^2 (x)}-C(\delta,k \delta,M)\right)\,dx\leq 0.
\end{equation}
As in the proof of Lemma \ref{blwbdd} we have $\dfrac{x\cdot \nu_{\rho}(x)}{\rho^2 (x)}-C(\delta,k \delta,D)\geq 0$, and since $ f>0$ a.e, equation  \eqref{eq:integraltauminusconstantnonpositive}
implies that
\begin{equation}\label{eq:constantoutsidesingularpts}
\dfrac{x\cdot \nu_{\rho}(x) }{\rho^2 (x)}=C(\delta,k \delta,M),\qquad \text{for a.e. $x\in \bar{\Omega}$}.
\end{equation}
We will show this implies $\rho$ is constant.
Notice that since $\sigma \in \mathcal A(\delta)$ we can apply inequality \eqref{eq:inequalityatnonsingularpoints} at each non singular point.
So from equation \eqref{eq:constantoutsidesingularpts} and the form of the constant $C$
\[
\rho(y)^2=\dfrac{y\cdot \nu_{\rho}(y) }{C(\delta,k \delta,M)}\geq \left( \dfrac{k\delta M}{1-c_\delta}\right)^2
\]
at a.e. $y\in \bar{\Omega}$. Since $\rho(x)\leq \dfrac{k\delta M}{1-c_\delta}$,
so $\rho(x)$ is constant a.e., and since $\rho$ is continuous, then is constant in $\bar{\Omega}$ . This is a contradiction. In fact, suppose $\rho$ is constant, then $\sigma$ is a piece of sphere with center $O$ and by Snell law every ray emitted from $O$ is reflected off by $\sigma$ to $O$. Recall that $m_0=\dfrac{\overrightarrow{OP_0}}{|OP_0|}$.
So if $x \in \tau_\sigma (P_0)$, then $x,m_0$ are collinear,  i.e., $x= \pm m_0$, therefore $|\tau_{\sigma}(P_0)|$=0, contradicting the case assumption.

Therefore \eqref{eq:Toveropenisbiggerthanmu} is proved.

Notice that if $\eta(P_0)>0$, then $\mu(P_0)>0$ and so the reflector overshoots.

\textbf {Case 2:} $\mu(P_0)=\int_{\tau_\sigma(P_0)}f(x)\,\dfrac{ x\cdot \nu_{\rho}(x) }{\rho^2 (x)}\,dx =0$.

This implies that $|\tau_\sigma(P_0)|=0$ and $\eta(P_0)=0$. Then for each $G$ open neighborhood of $P_0$ we have
\begin{align*}
\mu(G)=\int_{\tau_\sigma(G)}f(x)\,\dfrac{ x\cdot \nu_{\rho}(x)}{\rho^2 (x)}\,dx &=\int_{\tau_\sigma(G \setminus P_0)}f(x)\,\dfrac{x\cdot \nu_{\rho}(x)}{\rho^2 (x)}\,dx +\int_{\tau_\sigma(P_0)}f(x)\,\dfrac{ x\cdot \nu_{\rho}(x) }{\rho^2 (x)}\,dx\\
&=\eta (G \setminus P_0)\\
&=\eta(G).
\end{align*}
This identity also holds for every open set not containing $P_0$, and so for any open set in $D$. Since both measures $\mu$ and $\eta$ are outer regular
then they are equal. Therefore in this case the reflector doesn't overshoot.

\subsection{Optimization of $\delta$ in \eqref{probcondgeneral}}
Inequality \eqref{probcondgeneral} is a sufficient condition for the existence of a reflector and depends on $\delta>0$.
If we choose $\delta$ so that the right hand side of \eqref{probcondgeneral} 
is minimum, 
then 
we can choose the radiant intensity $f$ such that we have equality in \eqref{probcondgeneral}
and this will give the minimum energy required at the outset to have a reflector.
We calculate here the value of $\delta$ for which the RHS of \eqref{probcondgeneral} in minimum.
Indeed, the only conditions imposed on $\delta$ and $k$ are that $\delta >0$ and $k \geq \dfrac{1+c_\delta}{1-c_\delta}$.
We then have
$$C(\delta, k\delta, M)=\dfrac{(1-c_{\delta})^{3}}{(1+c_{\delta})(k\delta M)^{2}}\leq \dfrac{1}{M^2}\dfrac{(1-c_{\delta})^5}{\delta^2(1+c_\delta)^3}:=\dfrac{1}{M^2}\,r(\delta)$$
with $c_\delta = -\delta + \sqrt{1+\delta^2}$.
After some computations we get
\begin{itemize}
\item $\lim_{\delta \to 0} r(\delta)=0 .$
\item $\lim_{\delta \to \infty} r(\delta)=0 .$
\item $\dfrac{dr}{d\delta}= \dfrac{2(1-c_\delta)^4 (4c_{\delta}-1)}{\delta ^2 \sqrt{1+\delta ^2}(1+c_\delta)^4} .$
\end{itemize}
Therefore the absolute maximum of $r$ in $(0,+\infty)$ is attained when
$
c_\delta = \dfrac{1}{4}$, that is, when
$\delta =\dfrac{15}{8}$.
The corresponding value of $k$ is $k=\dfrac{1+\dfrac{1}{4}}{1-\dfrac{1}{4}}=\dfrac{5}{3}$; and the corresponding constant is
$$C\left(\dfrac{15}{8},\dfrac{25}{8}, M\right)=\dfrac{1}{M^2} r\left(\dfrac{15}{8}\right)= \dfrac{108}{3125 \,M^2}.$$
However, if we take into account the obstruction issue discussed in Subsection \ref{subset:physicalvisibilityissues}, we need $\delta> \delta_D$. 
Therefore, if $\delta_D < \dfrac{15}{8}$, 
then one can take $\delta=15/8$ in the argument above to obtain a constant that minimizes the outset energy required to obtain a reflector in $\mathcal A(15/8)$ that does not obstruct the target.
Otherwise, if $\delta_D \geq \dfrac{15}{8}$, 
for any $\delta >\delta_D$ there is a reflector in $\mathcal A(\delta)$ solving our problem and not obstructing the target but
the corresponding constant $C(\delta,k\delta,M)$ cannot be maximized.

\subsection{An observation}\label{sec:Generalproblem}
We remark that the existence results in this paper can be extended to the case when the function 
$\dfrac{x\cdot \nu(x)}{\rho(x)^2}$ in \eqref{eq:pbequ} is replaced by $F(x\cdot \nu(x),\rho(x))$ where $F(u,v)$ is a continuous and strictly positive function of two variables.
In particular, if $F$ is constant, then we recover the results of \cite{kochengin-oliker:nearfieldreflector}.
Similarly to Proposition \ref{mumeasure} we define the finite Borel measure on $D$ by 
$$\mu(E)=\int_{\tau_{\sigma}(E)}f(x)\,F\left(x\cdot \nu(x),\rho(x)\right)\,dx$$
when $F$ is defined in $\displaystyle{\left[\dfrac{1-c_\delta}{1+c_\delta},1\right]\times \left[\dfrac{\delta M}{1+c_\delta}, \max_{x\in \bar{\Omega}}\rho(x)\right]}$.
In addition, the stability property from Proposition \ref{weaklem} holds true also in this case 
when $F$ is continuous on $\displaystyle{\left[\dfrac{1-c_\delta}{1+c_\delta},1\right]\times \left[\dfrac{\delta M}{1+c_\delta}, b\right]}$.
With the set up of Lemma \ref{blwbdd}, we obtain the following inequality that replaces \eqref{cond} 
\begin{equation}\label{generalcondition}
\mu_w(D)=\int_{\bar{\Omega}}f(x)\,F\left(x\cdot \nu(x),\rho(x)\right)\,dx \geq \left(\min_{(u,v)\in K_{\delta,\delta'}}F(u,v)\right) \int_{\bar{\Omega}}f(x)\,dx
\end{equation}
where $F$ is continuous on $K_{\delta,\delta'}:=\left[\dfrac{1-c_\delta}{1+c_\delta},1\right]\times\left[\dfrac{\delta M}{1+c_\delta},\dfrac{\delta' M}{1-c_\delta}\right]$, and with equality in \eqref{generalcondition} if $F$ is constant.
Finally, the analogue of Theorem \ref{thm:solutionforgeneralmeasure} now follows under the assumption
\begin{equation}\label{generalprobcondgeneral}
\int_{\bar{\Omega}}f(x)\,dx \geq \dfrac{1}{\min_{(u,v)\in K_{\delta,k\delta}}F(u,v)}\eta(D),
\end{equation}
with $F$ continuous on $K_{\delta,k\delta}$.
If $F$ is constant and there is equality in \eqref{generalprobcondgeneral}, then there is no overshooting, i.e., 
\begin{equation*}
\int_{\tau_\sigma(E)}f(x)\,F\left(x\cdot \nu(x),\rho(x)\right)\,dx = \eta(E)
\end{equation*}
for each $E\subseteq D$ Borel set.

\setcounter{equation}{0}
\section{Derivation of the near field differential equation}\label{sec:NearFieldEquation}
Set $X=(x,x_3)$ a point on $\bar{\Omega} \subseteq S^{2}$, with $x=(x_1,x_2)$. Consider the reflector $\sigma=\{\rho(X)X\}$ solution of the near field problem from $\bar{\Omega}$ to $D$ with radiant intensity $f(x)$, with a measure $\eta$ defined on $D$ that is
absolutely continuous with respect to Lebesgue measure, i.e., $\eta= g dy$ with $g$ a positive function in $L^1(D)$. Let $\mathcal U=\{x:
(x,\sqrt{1-|x|^2})\in \bar{\Omega}\}$. 
If $\bar{\Omega}$ in the upper hemisphere $S_+^2 $, we can identify $\bar{\Omega}$ with $\mathcal U$
and assume that $\rho$ is a $C^2$ function on $\mathcal U$. 
If $X\in S^2$ and $Y \in S^{2}$ is the direction of the ray reflected off by $\sigma$, then by Snell's Law
\begin{equation}\label{Snellrefl}
Y=X-2\,\left(X \cdot \nu(X)\right)\,\nu(X)
\end{equation}
where $\nu $ is the outer unit normal to the reflector $\sigma$.
To derive the equation of the problem, we assume that $D \subseteq \{x_3=0\}$.
If $D$ is contained in a general surface, then the equation is deduced from this as in \cite{gutierrez-huang:nearfieldrefractor}. 
If the surface $\sigma$ reflects off the ray with direction $X$ into the point $Z \in D$, then 
$$Z=\rho(x)X+ |Z-\rho(x)X|Y.$$
Let $T$ be the map on $\mathcal U$,  $x\to Z$. Since $D \subseteq \{x_3=0\}$, $T(x)=(z_1,z_2,0)$.
If $dS_{\bar{\Omega}}$, $dS_{D}$ and $dS_{\mathcal U}$ denote the surface area elements of $\bar{\Omega}, D,\mathcal U$, respectively, then we have:
$$ |\det(DT)| =\dfrac{dS_{D}}{dS_{\mathcal U}} \qquad  dS_{\bar{\Omega}}=\dfrac{1}{\sqrt{1-|x|^2}} dS_U,$$
where $DT$ denotes the Jacobian of $T$.
From \eqref{eq:pbequ} we have that for every set $E \subseteq D$
$$\int_{\tau_{\sigma}(E)} f(x)\,\dfrac{x \cdot \nu(x)}{\rho(x)}\,dx \geq \int_E g(y)\,dy,$$
and as in \cite[Formula (8.11)]{gutierrez-huang:nearfieldrefractor} we get
\begin{equation}\label{simplequ}
|\det( DT)| \leq  \dfrac{f(x)\,\left(X\cdot \nu(X)\right)}{\sqrt{1-|x|^2}\, \rho^2 (x)\,  g(T(x))}.
\end{equation}
From \cite[Lemma 8.1]{gutierrez-mawi:refractorwithlossofenergy} we have
\begin{equation}\label{eq:formulasfornormal}
\nu= \dfrac{-\hat {D}\rho  +X\,\left(\rho+D\rho\cdot x\right)}{\sqrt{\rho^2+|D\rho|^2-(D\rho\cdot x)^2}},\qquad \text{ and }\qquad
X\cdot \nu = \dfrac{\rho}{\sqrt{\rho^2+|D\rho|^2-(D\rho\cdot x)^2}},
\end{equation}
where $D\rho(x) =\left(\partial _1 \rho(x),\partial_2 \rho(x)\right)$, $\hat{D}\rho(x) =\left(\partial _1 \rho(x),\partial_2 \rho(x),0\right)$.

Applying the calculations from \cite[Appendix]{gutierrez-huang:nearfieldrefractor} when $\kappa=1$ and $n=3$ we get
$$\det (DT) = \det \left(D^2\rho + \mathcal A(x,\rho,D\rho) \right)\,(2\rho)^2\, F\,\left(F+D\rho \cdot D_p F\right)$$
with
$F:=F(x,\rho(x),D\rho(x))$,  
$$F(x,u,p)=\dfrac{\dfrac{u}{\sqrt{u^2+|p|^2-(p\cdot x)^2}}}{-\sqrt{u^2+|p|^2-(p\cdot x)^2}+2(u+p\cdot x)\,\dfrac{u}{\sqrt{u^2+|p|^2-(p\cdot x)^2}}},$$
and
$$\mathcal A(x,\rho,D\rho)=\dfrac{1}{\rho\,(F+D\rho \cdot D_p F)}\,\left[(F+\rho F_u)D\rho\otimes D\rho+\rho D\rho \otimes D_x F\right].\footnote{For $a, b$ vectors in $\R^3$, $a\otimes b$ is the matrix $a^t b$.}$$
Replacing these formulas in inequality \eqref{simplequ},  we conclude that $\rho$ satisfies the following Monge-Amp\`ere type equation
\begin{align}\label{MongEqu}
&\left|\det\left(D^2\rho +\mathcal A\left(x,\rho(x),D\rho(x)\right)\right)\right|\\
&\leq \dfrac{f(x)}{4\, g(T(x))\, \sqrt{1-|x|^2}\, \left|F (F+D\rho \cdot D_p F)\right|\,\rho^3\sqrt{\rho^2+|D\rho|^2-(x\cdot D\rho)^2}}.\notag
\end{align}

\setcounter{equation}{0}
\section{The Far Field Case}\label{sec:Farfield}
The method used in the above sections can be applied similarly to construct a far field reflector, that is, the target $D$ is replaced by a set $\Omega^*\subseteq S^2$. 
Suppose radiation emanates from the origin $O$ with given radiant intensity $f(x)$ for each direction $x\in \Omega\subseteq S^2$, and we are given a Radon measure $\eta$ on $\Omega^*$. We want to construct a reflector surface $\sigma=\{\rho(x)x\}_{x\in \Omega}$ such that the radiation is reflected off by $\sigma$ into $\Omega^*$ 
such that
\begin{equation} \label{parapbequ}
\int_{\tau_\sigma (E)}f(x) \dfrac{x\cdot \nu(x)}{\rho(x)^2}\,dx\geq \eta(E)
\end{equation}
for each $E\subseteq \Omega^*$ and where $\tau_\sigma(E)$ is the collection of directions $x \in \Omega$ reflected off by $\sigma$ into 
$E$.
We assume the initial energy condition 
\begin{equation}\label{parainitial}
\int_{\Omega} f(x) \,dx \geq \dfrac{1}{C} \eta(\Omega^*)
\end{equation}
where $C$ is constant depending only on $\Omega, \Omega^*$ and on how close from the source we want to place the reflector.
In fact, given $m_0\in \supp(\eta)$ we will construct a reflector $\sigma$ such that we have equality in (\ref{parapbequ}) for every set $E$ not containing $m_0$.
As in the near field case, 
we next introduce the definition of solution and some properties now done with paraboloids of revolution.

\subsection{Paraboloids of revolution}
A paraboloid of revolution with focus $O$ and unit axis direction $m$ can be interpreted as an ellipsoid of eccentricity $1$ and foci $O$ and $P$, where $P$ is a point at infinity in the direction $m$. Therefore, by equations (\ref{eq:formulaforthefocalparameter}) and (\ref{ellipse}), the polar equation of a paraboloid $P_d(m)$ of focus O and axis direction $m$ and focal parameter $d$ is
\begin{equation}\label{def:paraofpolarradius}
\rho_d(x)= \dfrac{d}{1- x\cdot m } \qquad \text{with $x \in S^2 \setminus \{m\}$}
\end{equation}
In this case, it's immediate that for fixed $x$ and $m$, $\rho_d$ is strictly increasing in $d$.
By equation (\ref{nueq}), such a paraboloid has outer unit normal $\nu$ with
\begin{equation}\label{paranueq}
\nu(x)=\dfrac{x-m}{|x-m|}.
\end{equation}
We also have from the Snell law the well known fact that each ray emitted from $O$ with direction in $S^2 \setminus \{m\}$ is reflected off by $P_d(m)$ into the direction $m$.

\subsection{Far field reflectors and measures}

\begin{defi}
Let $\Omega\subseteq S^{2}$, with $|\partial \Omega|=0$. The surface $\sigma=\{\rho(x)x\}_{x\in \Bar{\Omega}}$ is a reflector from $\Bar{\Omega}$ to $\Bar{\Omega}^*$ if for each $x_0\in \bar{\Omega}$ there exists a paraboloid $P_d(m)$ with $m \in \Bar{\Omega}^* $ that supports $\sigma$ at $\rho(x_0)\,x_0$. That is, $P_d(m)$ is given by $\rho_d(x)=\dfrac{d}{1- \,x\cdot m}$ with $x \neq m$ and satisfies $\rho(x)\leq \rho_d(x)$ for all $x\in \Bar{\Omega}$ with equality at $x=x_0$.
\end{defi}
In order to prevent the degenerate case we will suppose that $x \cdot m \neq 1$ for all $x \in \Bar{\Omega}$ and $m \in \Bar{\Omega}^*$. By compactness there exist a uniform $0 < \delta <1 $ such that:
$$ x \cdot m \leq 1- \delta \qquad \forall x \in \Bar{\Omega}, m \in \Bar{\Omega}^*$$

Notice that each reflector is concave and therefore continuous.

The reflector mapping associated with a reflector $\sigma$ is given by
 \[
 \mathcal N_\sigma(x_0)=\{m\in \Bar{\Omega}^*: \text{there exists $P_d(m)$ supporting $\sigma$ at $\rho(x_0)x_0$}\};
 \]
 and the tracing mapping is
\[
\tau_\sigma(m)=\{x\in \Bar{\Omega}: m\in \mathcal N_\sigma(x) \}.
\]

\begin{prop}\label{impparaine}
Let $0<\delta<1$ and suppose $\Bar{\Omega} \cdot \Bar{\Omega}^* \leq 1- \delta$. 
If $P_d(m)$ is given in  ($\ref{def:paraofpolarradius}$) with $m \in \Bar{\Omega}^*$, then
\begin{equation}\label{paraineq}
\dfrac{d}{2} \leq \min_{x\in \Bar{\Omega}} \rho_d(x)\leq \max_{x\in \Bar{\Omega}} \rho_d(x) \leq \dfrac{d}{\delta}.
\end{equation}
\end{prop}

\begin{proof}
The proof follows from the fact that $-1 \leq x\cdot m \leq 1-\delta$.
\end{proof}

\begin{remark}\rm
We are interested in reflectors $\sigma$ such that they are at a positive distance from the origin.
Say we want $\rho(x)\geq a>0$ for all $x \in \Bar{\Omega}$.  To obtain this, by Proposition \ref{impparaine}  it is enough to pick $d$ such that $d\geq 2a$ and impose the condition that $d\geq 2a$ for each supporting paraboloid $P_d(m)$ to $\sigma$.
\end{remark}

\begin{defi}
For each $a>0$, we introduce the class
$
\mathcal A(a)$ of all reflectors $\sigma=\{\rho(x)x\}_{x\in \Bar{\Omega}}$ from $\Bar{\Omega}$ to $\Bar{\Omega}^*$  such that for each $x_0 \in \Bar{\Omega} $ there exist a supporting paraboloid $P_d(m)$ to $\sigma$ at $\rho(x_0)x_0$ with $m \in \Bar{\Omega}^*$ and focal parameter $d\geq 2 a$.
\end{defi}
\begin{prop}\label{paralipschitz}
If $\sigma=\{\rho(x)x\}$ is a far field reflector in $\mathcal A(a)$ and $\Bar{\Omega} \cdot \Bar{\Omega}^* \leq 1- \delta$, then $\rho$ is Lipschitz in $\Bar{\Omega}$, is bounded\footnote{The bound, and the Lipschitz constant is not necessarily uniform for all $\sigma\in \mathcal A(a)$.}, and the surface $\sigma$ is strictly convex.
\end{prop}

\begin{proof}
Strict convexity follows as in Proposition \ref{lipschitz}.
 Fix $x_0 \in \Bar{\Omega}$ and $m_0 \in \mathcal N_\sigma(x_0)$, then there exist a paraboloid $P_{d_0}(m_0)$ supporting $\sigma$ at $\rho(x_0)x_0$. 
Since for every $x \in \Bar{\Omega}$, there exist $P_d(m)$ with $m \in \Bar{\Omega}^*$ supporting $\sigma$ at $\rho(x)x$, then by Proposition \ref{impparaine}
$\dfrac{d}{2} \leq \rho_d(x) = \rho(x) \leq \rho_{d_0}(x) \leq \dfrac{d_0}{\delta},$
concluding that $d \leq \dfrac{2d_0}{\delta}$.
Now take $y \in \Bar{\Omega}$:
\begin{align*}
\rho(y) - \rho (x) &= \rho(y)-\rho_d(x)\leq \rho_d(y)-\rho_d(x)=\dfrac{d}{(1-  x\cdot m )(1- y\cdot m )} (y-x)\cdot  m\\
&\leq \dfrac{d}{\delta^2} |x-y|  
\leq \dfrac{2d_0}{\delta^3} |x-y|,
\end{align*}
and therefore the Proposition follows.
\end{proof}

%

We have from \cite[Lemma 1.1]{Wang:antenna} the following.
\begin{prop}\label{parapropmeasu}
If $\sigma$ is a reflector from $\Bar{\Omega}$ to $\Bar{\Omega}^*$,
and $A$ and $B$ are disjoint subsets of $\Bar{\Omega}^*$, then $\tau_\sigma(A) \cap \tau_\sigma(B)$ has Lebesgue measure zero.
\end{prop}


As a consequence of Proposition \ref{paralipschitz} and the proof of Proposition \ref{convprop} we obtain the following.
\begin{prop}\label{paraconvprop}
Suppose $\Bar{\Omega} \cdot \Bar{\Omega}^* \leq 1- \delta$.
Let $\sigma \in \mathcal A(a)$ and $N$ is the set of singular points of $\sigma$.
Suppose $\{x_n\}_{n=1}^\infty,x_0$ are in $\Bar{\Omega} \setminus N$ and $x_n\to x_0$.
If $P_{d_n}(m_n)$ and $P_{d_0}(m_0)$ are the corresponding supporting paraboloids to $\sigma$ at $x_n$ and $x_0$, and $\nu(x_n),\nu(x_0)$ are the corresponding unit normal vectors, then we have
\begin{enumerate}
\item $\lim_{n \to \infty} d_{n}= d_0$
\item $\lim_{n \to \infty} m_{n}= m_0$
\item $\lim_{n \to \infty} \nu(x_n)=\nu(x_0)$
\end{enumerate}
\end{prop}

Using Propositions \ref{parapropmeasu}, \ref{paraconvprop}, and proceeding as Section \ref{sec:reflectorsandreflectormeasures}, we get the following results.

\begin{prop}\label{paramumeasure}
Let $f\in L^1(\Bar{\Omega})$ be non negative, $\Bar{\Omega} \cdot \Bar{\Omega}^* \leq 1- \delta$, and let $\sigma=\{\rho(x)x\}_{x\in \bar \Omega}$ be a reflector in $\mathcal A(a)$ for some $a>0$.
We define
\[
\mu(E)=\int_{\tau_{\sigma}(E)} f(x) \dfrac{ x\cdot \nu(x)}{\rho ^2(x)}\,dx
\] for each Borel set $E\subseteq \Bar{\Omega}^*$.
Then $\mu$ is a finite Borel measure on $\Bar{\Omega}^*$.
\end{prop}

\begin{prop} \label{paraweaklem}
Suppose $\Bar{\Omega} \cdot \Bar{\Omega}^* \leq 1- \delta$.
Let $\sigma_n$ be a sequence of reflectors in $\mathcal A(a)$ for some fixed $a>0$, where $\sigma_n=\{ \rho_n (x)x \}$ such that $\rho_n(x)\leq b$ for all $x \in \Bar{\Omega}$ and $\rho_n$ converges point-wise to $\rho$ in $\Bar{\Omega}$.
Let $\sigma=\{\rho(x)x\}$. Then we have:
\begin{enumerate}
\item $\sigma\in \mathcal A(a)$.
\item If $\mu$ is the reflector measure corresponding to $\sigma$, then $\mu_n$ converges weakly to $\mu$ .
\end{enumerate}
\end{prop}

\begin{cor}\label{paracontinui}
If $\Omega^*=\{m_1,m_2,...,m_N\}$ and $\sigma_n,\sigma,\mu_n,\mu$ are as in the above proposition, then
$$\lim_{n \to +\infty} \mu_n(m_i)=\mu(m_i).$$
\end{cor}

\subsection{Solution of the problem}

\begin{defi}\label{FarReflector}
Let $0<\delta <1$, and let $\Omega \subseteq S^2$, and $\Omega ^*=\{m_1,m_2,...,m_N\}$ be  such that $\Bar{\Omega} \cdot \Bar{\Omega}^* \leq 1-\delta$.
Let $d_1,\cdots ,d_N$ be positive numbers and $w=(d_1,d_2,\cdots ,d_n)$.
We define the reflector $\sigma=\{\rho_w(x)x\}_{x\in \Bar{\Omega}}$ by
\[
\rho_w(x)=\min_{1\leq i \leq N} \rho_{d_i}(x),
\]
with $\rho_{d_i}(x)=\dfrac{d_i}{1- x \cdot m_i}$.
\end{defi}

\begin{lem}\label{parablwbdd}
Let $0<a '$, and let $\{\rho_w(x)x\}$ be the reflector with $w=(d_1,\cdots ,d_N)$, where $d_1\leq a'$.
If $f\in L^1(\Bar{\Omega})$ and $f>0$ a.e.,
then
\begin{equation}\label{paracond}
\mu_w(D)=\int_{\Bar{\Omega}}f(x)\dfrac{x\cdot \nu_w(x)}{\rho_w ^2 (x)}dx > C(a',\delta)
\int_{\Bar{\Omega}} f(x)dx
\end{equation}
where $C(a',\delta)$ is a constant depending only on $a'$ and $\delta$.
\end{lem}

\begin{proof}
From Proposition \ref{impparaine}, $\rho_{d_1}(x) \leq \dfrac{d_1}{\delta} \leq \dfrac{a'}{\delta}$,
and so by the Definition \ref{FarReflector} $\rho_w(x)\leq \dfrac{a'}{\delta}$.
From Proposition \ref{paralipschitz}, the set of singular points of the reflector $\rho_w$ has measure zero.
So for each $x \in \Bar{\Omega}$ not a singular point, there exists $1\leq i \leq N$ such that
\begin{equation*}
 x\cdot\nu_w(x) = x\cdot \nu_{d_i}(x)
= \dfrac{1-  x\cdot m_i}{|x- m_i|}\geq \dfrac{\delta}{2},\qquad
\text{by \eqref{paranueq}.}
\end{equation*}
Therefore $\dfrac{ x\cdot \nu_w(x)}{\rho_w ^2 (x)}\geq \dfrac{\delta^3}{2a'^2}:=C(a ',\delta)$ for a.e. $x$, and so
$$\mu_w(D)=\int_{\Bar{\Omega}}f(x)\dfrac{x\cdot \nu_w(x)}{\rho_w ^2 (x)}dx \geq C(a ',\delta)\int_{\Bar{\Omega}} f(x)dx.$$
To prove that the inequality is strict we proceed as in Lemma \ref{blwbdd}.
\end{proof}

The statements of Lemma \ref{compare} and Corollary \ref{corcomp} hold true in the far field case and we therefore obtain the solution in the discrete case.
%
%
%
%
%

\begin{theorem}\label{paradiscretethmgeq}
Let $\Omega\subseteq S^2,0 < \delta <1$,
and $f\in L^1(\bar{\Omega})$ such that $f>0$ a.e.
Suppose $g_1,g_2,\cdots ,g_N$ are positive numbers with $N > 1$, and
$\Omega^*=\{m_1,m_2,....,m_N\}\subseteq S^2$  be such that $\Bar{\Omega} \cdot \Bar{\Omega}^* \leq 1-\delta$.
Define the measure $\eta$ on $\Bar{\Omega}^*$ by $\eta=\sum_{i=1}^{N} g_i \delta_{P_i}$.
Fix $a > 0$ and $a' \geq \dfrac{4a}{\delta}$, and suppose that
\begin{equation}\label{paraprobcond}
\int_{\Bar{\Omega}}f(x)dx \geq \dfrac{1}{C(a',\delta)} \eta(\Omega^*),
\end{equation}
where $C(a',\delta)=\dfrac{\delta^3}{2a'^2}$.

Then there exists a reflector $\Bar{w}=(a',\bar {d_2},\cdots ,\Bar{d_{N}})$ with $\Bar d_{i}\geq 2a$ for $1 \leq i \leq N$ satisfying
\begin{enumerate}
\item $\Bar{\Omega}=\bigcup_{i=1}^{N} \tau_{\Bar{\sigma}}(m_{i})$.
\item $ \Bar{\mu}(P_i)= g_i$ for $ 2 \leq i \leq N$, where $\Bar{\mu}$ is the reflector measure corresponding to $\Bar{w}$; and
\item $\Bar{\mu}(P_1) > g_1$.
\end{enumerate}

\end{theorem}

For the general case we have the following.
\begin{theorem}\label{thm:parasolutionforgeneralmeasure}
Let $\Omega,\Omega^*\subseteq S^2$, such that $\Bar{\Omega} \cdot \Bar{\Omega}^* \leq 1-\delta$ for some $0<\delta<1$.
Let $f\in L^1(\Bar{\Omega})$ be such that $f>0$ a.e, and let $\eta$ be a Radon measure on $\Bar{\Omega}^*$.
Given $a>0$, and $a' \geq \dfrac{4a}{\delta}$ we assume that
\begin{equation}\label{paraprobcondgeneral}
\int_{\bar \Omega}f(x)dx \geq \dfrac{1}{C(a',\delta)} \eta(\Bar{\Omega}^*),
\end{equation}
where $C(a',\delta)=\dfrac{\delta^3}{2a'^2}$.

Then given $m_0\in \supp (\eta)$, the support of the measure $\eta$,
there exists a reflector $\sigma=\{\rho(x)x\}_{x\in \Bar{\Omega}}$ from $\Bar{\Omega}$ to $\Bar{\Omega}^*$ such that
\[
\eta(E)\leq \int_{\tau_\sigma(E)}f(x)\,\dfrac{x\cdot \nu_{\rho}(x)}{\rho^2 (x)}dx,
\]
for each Borel set $E\subseteq \Bar{\Omega}^*$, and
\[
\eta(E)= \int_{\tau_\sigma(E)}f(x)\,\dfrac{x\cdot \nu_{\rho}(x)}{\rho^2 (x)}dx,
\]
for each Borel set $E\subseteq \Bar{\Omega}^*$ with $m_0\notin E$.
In addition, the reflector $\sigma=\{\rho(x)x\}$ belongs to $\mathcal A(a)$, and
$a \leq \rho(x)\leq \dfrac{a'}{\delta}$ for all $x$.
\end{theorem}

\begin{proof}
Using the proof of Proposition \ref{paralipschitz}, we have that every reflector solution in Theorem \ref{paradiscretethmgeq} with $w=(a',d_2,\cdots,d_n)$ and $a'\geq \dfrac{4a}{\delta}, d_i \geq 2a$, satisfies
$$a\leq \rho(x)\leq \dfrac{a'}{\delta} \quad \text{ and }\quad |\rho(x)-\rho(y)|\leq \dfrac{2a'}{\delta^3}|x-y| \text{          for all $ x,y \in \Bar{\Omega}$}.$$ Hence Arzel\'a-Ascoli theorem can be applied and the proof follows as in Theorem \ref{thm:solutionforgeneralmeasure}.
\end{proof}

\subsection{Overshooting discussion}
As in the near field case, there exists a far field reflector $\sigma_0$ such that:
$$\mu_0(m_0)=\inf\left\{\int_{\tau_{\sigma}(m_0)}f(x)\dfrac{x\cdot \nu(x)}{\rho^2(x)}dx: \text{$\sigma$ a reflector as in Theorem \ref{thm:parasolutionforgeneralmeasure}}\right\} .$$
Moreover, for every open set $G$ containing $m_0$ we have
$\mu_0(G)>\eta(G)$ if $\mu_0(m_0)>0$, and
$\mu_0(G)=\eta(G)$ if $\mu_0(m_0)=0$.


\subsection{Derivation of the far field differential equation}\label{sec:FarFieldEquation}
We use the notation from Section \ref{sec:NearFieldEquation}.
Let $\sigma=\{\rho(X)X\}$ be a far field reflector from $\bar{\Omega}$ to $\bar{\Omega}^*$ with emitting radiance intensity $f(x)$, and $\eta= g \,dx$ with $g$ positive function in $L^1(\bar{\Omega}^*)$. 
Recall $\mathcal U=\{x:
(x,\sqrt{1-|x|^2})\in \bar{\Omega}\}$ and suppose $\rho$ is $C^2$. 
Keeping in mind 
\eqref{Snellrefl},
we now let $T$ be the map on $\mathcal U$ defined by $x\to Y$.
%
If $DT$ is the Jacobian of $T$, then as in the near field case we can deduce that
\begin{equation}\label{parasimplequ}
|\det( DT)| \leq  \dfrac{f(x)\,(X\cdot \nu(X))}{\sqrt{1-|x|^2}\, \rho^2 (x) \, g(T(x))}.
\end{equation}
%
We also have from the proof of \cite[Theorem 8.2]{gutierrez-mawi:refractorwithlossofenergy} when $\kappa=1$ and $n=3$ that
$$\det (DT) = \dfrac{1}{w \sqrt{1-|x|^2}}\, \det (D^2\rho + C^{-1}\,B)\, \det (C)$$
with
\begin{align*}
C(x) &= x\otimes D_p w +h\,I+D\rho \otimes D_p h\\
B(x) &= w\,I+x\otimes D_x w+ w_z\, x \otimes D\rho +D\rho \otimes D_x h +h_z\, D\rho \otimes D\rho\\
\det (C)&= h^2\,(1-|x|^2)\,\left(1-h^{-1}\left(\dfrac{\rho}{1-|x|^2}\,x-D\rho\right)\cdot D_p h\right)
\end{align*}
where $h:=h(x,\rho(x),D\rho(x)),w:=w(x,\rho(x),D\rho(x))$, and 
$$h(x,z,p)=\dfrac{2\,z}{z^2+|p|^2-(p\cdot x)^2}, \qquad w(x,z,p)=1-h(x,z,p)(z+p\cdot x).$$
Replacing these values in equation (\ref{parasimplequ}), and using \eqref{eq:formulasfornormal} we conclude that $\rho$ satisfies the following Monge-Amp\`ere type equation
\begin{align}\label{paraMongEqu}
&|\det(D^2\rho +C^{-1} B)|\\
&\leq \dfrac{f(x)\,|w|}{g(T(x))\,(1-|x|^2)\,h^2\left|1-h^{-1}\left(\dfrac{\rho}{1-|x|^2}x-D\rho\right)
\cdot D_p h\right|\rho\sqrt{\rho^2+|D\rho|^2-(D\rho \cdot x)^2}}.\notag
\end{align}

\providecommand{\bysame}{\leavevmode\hbox to3em{\hrulefill}\thinspace}
\providecommand{\MR}{\relax\ifhmode\unskip\space\fi MR }
\providecommand{\MRhref}[2]{%
  \href{http://www.ams.org/mathscinet-getitem?mr=#1}{#2}
}
\providecommand{\href}[2]{#2}

\end{document}